\documentclass[12pt]{amsart}
\usepackage{amssymb,amsmath,tabularx,mathrsfs,mathbbol,yfonts,upgreek}
\usepackage{amsthm,verbatim,comment}
\usepackage[bookmarks=true]{hyperref}
\usepackage{pstricks,pst-node,pst-plot}
\usepackage{geometry}
\usepackage{stmaryrd}
\usepackage{paralist}
\usepackage[all]{xy}
\usepackage{array}
\usepackage{url}
\usepackage{longtable}
\numberwithin{equation}{section}

\DeclareSymbolFontAlphabet{\mathbb}{AMSb}
\DeclareSymbolFontAlphabet{\mathbbl}{bbold}

\newtheorem{thm}{Theorem}[section]

\newtheorem{lem}[thm]{Lemma}

\newtheorem{cor}[thm]{Corollary}

\theoremstyle{definition}

\newtheorem{nota}[thm]{Notation}
\newtheorem{eg}[thm]{Example}
\newtheorem{rem}[thm]{Remark}

\newtheorem*{rem*}{Remarks}

\newtheoremstyle{case}{}{}{}{}{}{:}{ }{}
\theoremstyle{case}

\newcommand{\F}{\mathbb{F}}

\title[On elementary abelian 2-hypergroups]{On elementary abelian 2-hypergroups}
\begin{document}
\author{Yu Jiang}
\address[Y. Jiang]{School of Mathematical Sciences, Anhui University (Qingyuan Campus), No. 111, Jiulong Road, Hefei, 230601, China}
\email[Y. Jiang]{jiangyu@ahu.edu.cn}
\begin{abstract}
A hypergroup is called an elementary abelian $2$-hypergroup if it is a constrained direct product of the closed subsets of two elements. In this paper, the elementary abelian 2-hypergroups are studied. All closed subsets and all strongly normal closed subsets of the elementary abelian 2-hypergroups are determined. The numbers of all closed subsets and all strongly normal closed subsets of the elementary abelian 2-hypergroups are given. A criterion for the isomorphic closed subsets of the elementary abelian 2-hypergroups is displayed. The automorphism groups of all closed subsets of the elementary abelian 2-hypergroups are presented.
\end{abstract}
\maketitle
\vspace{-1.5em}
\section{Introduction}
Hypergroups, introduced by Marty in \cite{M}, are known as natural generalizations of groups. In particular, many important objects from the group theory have already been generalized to the corresponding objects from the hypergroup theory (see \cite{Z3}).

The subgroups of groups are important objects in the structure theory of groups. The closed subsets of hypergroups generalize the subgroups of groups (see \cite{Z3}). In particular, the closed subsets of hypergroups can describe the quotient structures of hypergroups. In general, the closed subsets of a hypergroup are poorly understood.

The closed subsets of hypergroups have been researched by many authors. In \cite{FZ}, French and Zieschang studied the residually thin closed subsets of hypergroups. In \cite{GZ}, Guo and Zhang studied the nilpotent closed subsets of hypergroups. In \cite{Z1,Z2,Z3}, Zieschang studied the closed subsets of hypergroups generated by the involutions. In \cite{Z1,Z3}, Zieschang studied the isomorphisms and the automorphisms of hypergroups.

A hypergroup is called an elementary abelian $2$-hypergroup if it is a constrained direct product of the closed subsets of two elements. In this paper, we determine all closed subsets and all strongly normal closed subsets of the elementary abelian 2-hypergroups (see Theorems \ref{T;TheoremA} and \ref{T;TheoremB}, respectively). We give the numbers of all closed subsets and all strongly normal closed subsets of the elementary abelian 2-hypergroups (see Theorem \ref{T;TheoremC}). We display a criterion for the isomorphic closed subsets of the elementary abelian 2-hypergroups (see Theorem \ref{T;TheoremD}). Furthermore, we also present the automorphism groups of all closed subsets of the elementary abelian 2-hypergroups (see Theorem \ref{T;TheoremE}). These stated main results of this paper contribute to investigating the structure theory of finite commutative hypergroups.

The outline of this paper is as follows: In Section 2, we list the basic notation and the needed preliminaries. The first three main results of this paper are contained in Section 3. The remaining two main results of this paper are contained in Section 4.
\section{Basic notation and preliminaries}
For a general theory on hypergroups, the reader may refer to the monograph \cite{Z3}.
\subsection{Conventions}
Let $\mathbb{N}$ be the set of all natural numbers. Set $\mathbb{N}_0=\mathbb{N}\cup\{0\}$. If $p, q\in\mathbb{N}_0$, set $[p, q]=\{a: a\in\mathbb{N}_0, p\leq a\leq q\}$.
Let $\alpha^{-1}$ be the inverse of a bijection $\alpha$. If $\alpha$ is a map with the domain $\mathbb{H}$ and $p\in\mathbb{H}$, let $p^\alpha$ be the image of $p$ under $\alpha$ and
$\mathbb{G}^\alpha\!=\!\{a^\alpha: a\in\mathbb{G}\}$ for any $\mathbb{G}\subseteq\mathbb{H}$. The composition of maps is from left to right. If $p\in\mathbb{N}$,
let $\mathbb{S}_{p}$ and $\mathbb{GL}(p, q)$ be the symmetric group on $[1, p]$ and the general linear group of degree $p$ over a field of $q$ elements, respectively.
Let $\mathbb{S}_0$ and the general linear group of degree zero over a field equal the trivial group.
If $p$ is a prime and $q\in\mathbb{N}_0$, the automorphism group of an elementary abelian $p$-group is isomorphic to $\mathbb{GL}(q, p)$.
If $p$ is a prime and $q, r\in \mathbb{N}_0$, the number of all subgroups of $p$-rank $r$ of an elementary abelian $p$-group of $p$-rank $q$ equals exactly the Gauss coefficient $\binom{q}{r}_p$.
\subsection{Hypergroups}
Let $\mathbb{H}$ denote a nonempty set. Each binary operation from the cartesian product $\mathbb{H}\!\times\!\mathbb{H}$ to the power set of $\mathbb{H}$ is called a hypermultiplication
on $\mathbb{H}$. Fix a hypermultiplication $\circ$ on $\mathbb{H}$. For any $p, q\!\in\!\mathbb{H}$, the image of the pair $(p, q)$ under $\circ$ is denoted by $pq$.
For any $\mathbb{F}, \mathbb{G}\!\subseteq\!\mathbb{H}$, use $\mathbb{F}\mathbb{G}$ to denote $\{a:\exists\ b\in\mathbb{F}, \exists\ c\in\mathbb{G}, a\in bc\}$.
If $p\!\in\!\mathbb{N}\setminus\{1\}$ and $\mathbb{G}_1, \mathbb{G}_2, \ldots, \mathbb{G}_p\!\!\subseteq\!\!\mathbb{H}$, set
$\mathbb{G}_1\mathbb{G}_2\cdots\mathbb{G}_p\!=\!(\mathbb{G}_1\mathbb{G}_2\cdots \mathbb{G}_{p-1})\mathbb{G}_p$ inductively.
For any $p\in\mathbb{H}$ and $\mathbb{G}\subseteq\mathbb{H}$, let $p\mathbb{G}$ and $\mathbb{G}p$
be $\{p\}\mathbb{G}$ and $\mathbb{G}\{p\}$, respectively. Fix a map $*$ from $\mathbb{H}$ to $\mathbb{H}$. Call $\mathbb{H}$ a hypergroup if the following conditions hold together:
\begin{enumerate}[(H1)]
\item The hypermultiplication $\circ$ is associative, i.e., $(pq)r\!=\!p(qr)$ for any $p, q, r\in \mathbb{H}$;
\item There exists a unique $e\in\mathbb{H}$ such that $ep=\{p\}$ and $pe=\{p\}$ for any $p\in\mathbb{H}$;
\item The conditions $p\!\in\! qr$, $q\!\in\! pr^*$, $r\!\in\! q^*p$ are pairwise equivalent for any $p, q, r\!\in\!\mathbb{H}$.
\end{enumerate}

From now on, $\mathbb{H}$ denotes a fixed hypergroup. If $\mathbb{E}, \mathbb{F}, \mathbb{G}\!\subseteq\!\mathbb{H}$, then $(\mathbb{E}\mathbb{F})\mathbb{G}=\mathbb{E}(\mathbb{F}\mathbb{G})$ by (H1).
If $p\in\mathbb{N}$ and $\mathbb{G}\subseteq\mathbb{H}$, set $\mathbb{G}^0=\{e\}$ and $\mathbb{G}^p=\mathbb{G}^{p-1}\mathbb{G}$ inductively. So $p^q$ is defined for any $p\in\mathbb{H}$ and $q\in\mathbb{N}_0$.
If $p\in \mathbb{H}$, notice that $(p^*)^*\!=\!p$ by (H2) and (H3). For any $p\in \mathbb{H}$, call $p$ a symmetric element of $\mathbb{H}$ if $p^*\!=\!p$.
If $\mathbb{G}\subseteq\mathbb{H}$, let $\mathrm{Sym}(\mathbb{G})$ be the set of all symmetric elements of $\mathbb{H}$ in $\mathbb{G}$. Hence $e\in\mathrm{Sym}(\mathbb{H})$ by (H2) and (H3).

For any $p\in\mathbb{H}$, $e\!\in\! p^*p\cap pp^*$ by (H2) and (H3). By (H2), $pq\!\neq\!\varnothing$ for any $p, q\in\mathbb{H}$. For any $p\!\in\!\mathbb{H}$, call $p$ a thick element of $\mathbb{H}$ if $p^*p\!\neq\!\{e\}$.
For any $p\!\in\!\mathbb{H}$, call $p$ a thin element of $\mathbb{H}$ if $p^*p=\{e\}$. For any $\mathbb{G}\subseteq\mathbb{H}$, let $\mathrm{O}_\vartheta(\mathbb{G})$ be the set of all thin elements of $\mathbb{H}$ in $\mathbb{G}$.
Therefore $e\in \mathrm{O}_\vartheta(\mathbb{H})$ as (H2) holds. The following lemmas are necessary:
\begin{lem}\label{L;Lemma2.1}\cite[Lemma 1.4.3 (i)]{Z3}
Assume that $p\!\in\!\mathbb{H}$ and $q\!\in\!\mathrm{O}_\vartheta(\mathbb{H})$. Then $|pq|=1$.
\end{lem}
\begin{lem}\label{L;Lemma2.2}\cite[Lemma 1.4.3 (iii)]{Z3}
Assume that $p, q\in \mathrm{O}_\vartheta(\mathbb{H})$. Then $pq\subseteq \mathrm{O}_\vartheta(\mathbb{H})$.
\end{lem}
For any $\mathbb{F}, \mathbb{G}\!\subseteq\!\mathbb{H}$, $(\mathbb{F}\mathbb{G})^*\!=\!\mathbb{G}^*\mathbb{F}^*$ by (H3).
If $\mathbb{G}\!\subseteq\!\mathbb{H}$ and $\mathbb{G}\neq\varnothing$, call $\mathbb{G}$ a closed subset of $\mathbb{H}$ if
$\mathbb{G}^*\mathbb{G}\subseteq\mathbb{G}$. So $\{e\}$ and $\mathbb{H}$ are always closed subsets of $\mathbb{H}$. So
$e\!\in\!\mathbb{G}$, $\mathbb{G}^*\!=\!\mathbb{G}$, $\mathbb{G}^2\subseteq\mathbb{G}$ for any a closed subset $\mathbb{G}$ of $\mathbb{H}$. The following lemmas are necessary:
\begin{lem}\label{L;Lemma2.3}\cite[Lemma 2.1.5]{Z3}
Assume that $\mathbb{F}$ and $\mathbb{G}$ are closed subsets of $\mathbb{H}$. Then $\mathbb{F}\mathbb{G}$ is a closed subset of $\mathbb{H}$ if and only if $\mathbb{F}\mathbb{G}=\mathbb{G}\mathbb{F}$.
\end{lem}
\begin{lem}\label{L;Lemma2.4}\cite[Lemma 2.1.6 (ii)]{Z3}
Assume that $\mathbb{F}$ and $\mathbb{G}$ are closed subsets of $\mathbb{H}$. Then $\mathbb{F}\cap\mathbb{G}\!=\!\{e\}$ if and only if, for any $p\in\mathbb{F}\mathbb{G}$, $p\in qr$ for unique $q\in\mathbb{F}$ and $r\in\mathbb{G}$.
\end{lem}
For any $p\in\mathbb{H}$, call $p$ an involution of $\mathbb{H}$ if $p\neq e$ and $\{e, p\}$ is a closed subset of $\mathbb{H}$.
Hence $p\in\mathrm{Sym}(\mathbb{H})$ for any an involution $p$ of $\mathbb{H}$. The following lemma is necessary:
\begin{lem}\label{L;Lemma2.5}
Assume that $p$ and $q$ are distinct involutions of $\mathbb{H}$. Assume that $p$ or $q$ is a thick element of $\mathbb{H}$. Assume that $pq=\{r\}$. Then $r$ is not an involution of $\mathbb{H}$.
\end{lem}
\begin{proof}
Assume that $r$ is an involution of $\mathbb{H}$. So $pq=qp$ and $r^2=pqpq=p^2q^2$ by (H1).
So $r\in \{p, q\}$ by (H1). Hence $p\in pq$ or $q\in pq$. Therefore $p=q$ by (H3). This is a contradiction. The desired lemma thus follows from the above discussion.
\end{proof}
For any $p\in\mathbb{N}$ and closed subsets $\mathbb{G}_1, \mathbb{G}_2, \ldots, \mathbb{G}_p$ of $\mathbb{H}$,
notice that the intersection of $\mathbb{G}_1, \mathbb{G}_2, \ldots, \mathbb{G}_p$ is also a closed subset of $\mathbb{H}$.
For any distinct closed subsets $\mathbb{F}, \mathbb{G}$ of $\mathbb{H}$, call $\mathbb{F}$ a maximal closed subset of $\mathbb{G}$ if $\mathbb{F}\subseteq\mathbb{G}$
and there is not a closed subset $\mathbb{E}$ of $\mathbb{H}$ such that $\mathbb{F}\subseteq\mathbb{E}\subseteq\mathbb{G}$ and $\mathbb{E}\notin\{\mathbb{G}, \mathbb{F}\}$.
The intersection of all maximal closed subsets of a closed subset $\mathbb{G}$ of $\mathbb{H}$ is called the Frattini closed subset of $\mathbb{G}$ (see \cite{Z1}).

For any $\mathbb{G}\subseteq\mathbb{H}$, the intersection of all closed subsets of $\mathbb{H}$ containing $\mathbb{G}$ is denoted by $\langle\mathbb{G}\rangle$. Notice that $\langle\mathbb{G}\rangle$ is a closed subset of $\mathbb{H}$ for any $\mathbb{G}\subseteq\mathbb{H}$. For any $p\in\mathbb{N}$ and $\{q_1, q_2, \ldots, q_p\}\!\subseteq\!\mathbb{H}$,
define $\langle q_1, q_2, \ldots, q_p\rangle\!=\!\langle \{q_1, q_2, \ldots, q_p\}\rangle$. For any an involution $p$ of $\mathbb{H}$,
it is obvious to notice that $\langle p\rangle=\{e, p\}$. The following lemmas are necessary:
\begin{lem}\label{L;Lemma2.6}\cite[Theorem 1.4.4 (i)]{Z1}
Assume that $\mathbb{G}$ is a closed subset of $\mathbb{H}$ and $\mathbb{F}$ is a subset of the Frattini closed subset of $\mathbb{G}$. If $\mathbb{E}\!\subseteq\!\mathbb{G}$ and $\mathbb{G}\!=\!\langle \mathbb{E}\cup\mathbb{F}\rangle$, then $\mathbb{G}=\langle\mathbb{E}\rangle$.
\end{lem}
\begin{lem}\label{L;Lemma2.7}\cite[Lemma 2.3.4 (i)]{Z3}
Assume that $\mathbb{G}\subseteq\mathbb{H}$ and $\mathbb{G}^*=\mathbb{G}$. Then
$$\langle\mathbb{G}\rangle=\bigcup_{p\in\mathbb{N}_0}\mathbb{G}^p.$$
\end{lem}
For any a closed subset $\mathbb{G}$ of $\mathbb{H}$ and $\mathbb{F}\!\!\subseteq\!\!\mathbb{G}$, call $\mathbb{F}$ a generating subset of $\mathbb{G}$ if $\mathbb{G}=\langle\mathbb{F}\rangle$. For any a closed subset $\mathbb{G}$ of $\mathbb{H}$, notice that $\mathbb{G}$
itself is a generating subset of $\mathbb{G}$. For any a closed subset $\mathbb{G}$ of $\mathbb{H}$, call $\mathbb{G}$ a finitely generated closed subset of $\mathbb{H}$ if a generating subset of $\mathbb{G}$ is a finite subset of $\mathbb{G}$. For any a closed subset $\mathbb{G}$ of $\mathbb{H}$, notice that $\mathbb{G}$ is a finitely generated closed subset of $\mathbb{H}$ if $\mathbb{G}$ is a finite subset of $\mathbb{H}$.

For any a closed subset $\mathbb{G}$ of $\mathbb{H}$ and $\mathbb{F}\subseteq\mathbb{G}$, call $\mathbb{F}$ a minimal generating subset of $\mathbb{G}$ if $\mathbb{G}=\langle\mathbb{F}\rangle$ and $\mathbb{G}\!\neq\!\langle\mathbb{F}\setminus\{p\}\rangle$ for any $p\in \mathbb{F}$. For any a finitely generated closed subset $\mathbb{G}$ of $\mathbb{H}$, notice that $\mathbb{G}$ has a finite minimal generating subset of $\mathbb{G}$. For any a finitely generated closed subset $\mathbb{G}$ of $\mathbb{H}$ and $\mathbb{F}\subseteq\mathbb{G}$, call $\mathbb{F}$ a basis of $\mathbb{G}$ if $\mathbb{F}$ is a minimal generating subset of $\mathbb{G}$ with the smallest cardinality. Then the cardinality of a basis of a finitely generated closed subset $\mathbb{G}$ of $\mathbb{H}$ is called the dimension of $\mathbb{G}$.

For any a closed subset $\mathbb{G}$ of $\mathbb{H}$, call $\mathbb{G}$ a commutative closed subset of $\mathbb{H}$ if $pq=qp$ for any $p, q\in\mathbb{G}$. For any a closed subset $\mathbb{G}$ of $\mathbb{H}$, notice that $\mathbb{G}$ is a commutative closed subset of $\mathbb{H}$ if $\mathbb{G}=\mathrm{Sym}(\mathbb{G})$. Then $\{e\}$ is a commutative closed subset of $\mathbb{H}$. Notice that $\langle p\rangle$ is a commutative closed subset of $\mathbb{H}$ for any an involution $p$ of $\mathbb{H}$.
Call $\mathbb{H}$ a commutative hypergroup if $\mathbb{H}$ itself is a commutative closed subset of $\mathbb{H}$.

For any closed subsets $\mathbb{F}, \mathbb{G}$ of $\mathbb{H}$, call $\mathbb{F}$ a normal closed subset of $\mathbb{G}$ if $\mathbb{F}\subseteq\mathbb{G}$ and $\mathbb{F}p=p\mathbb{F}$ for any $p\in\mathbb{G}$. By (H2), $\{e\}$ and $\mathbb{H}$ are always normal closed subsets of $\mathbb{H}$. Notice that $\mathbb{F}\mathbb{G}=\mathbb{G}\mathbb{F}$ for any normal closed subsets $\mathbb{F}, \mathbb{G}$ of $\mathbb{H}$. All closed subsets of $\mathbb{H}$ are precisely all normal closed subsets of $\mathbb{H}$ if $\mathbb{H}$ is a commutative hypergroup.

For any closed subsets $\mathbb{F}, \mathbb{G}$ of $\mathbb{H}$, call $\mathbb{F}$ a strongly normal closed subset of $\mathbb{G}$ if
$p^*\mathbb{F}p\subseteq\mathbb{F}\subseteq\mathbb{G}$ for any $p\in \mathbb{G}$. For any a closed subset $\mathbb{G}$ of $\mathbb{H}$,
notice that $\mathbb{G}$ itself is always a strongly normal closed subset of $\mathbb{G}$. If $p\in\mathbb{N}$ and $\mathbb{F}_1, \mathbb{F}_2, \ldots, \mathbb{F}_p$ are strongly normal closed subsets of a closed subset $\mathbb{G}$ of $\mathbb{H}$, the intersection of $\mathbb{F}_1, \mathbb{F}_2, \ldots, \mathbb{F}_p$ is a strongly normal closed subset of $\mathbb{G}$. For any $\mathbb{G}\subseteq\mathbb{H}$, use $\mathrm{O}^\vartheta(\mathbb{G})$ to denote the intersection of all strongly normal closed subsets of $\langle\mathbb{G}\rangle$. For any $\mathbb{G}\subseteq\mathbb{H}$, $\mathrm{O}^\vartheta(\mathbb{G})$ is always a strongly normal closed subset of $\langle\mathbb{G}\rangle$. The following lemmas are necessary:
\begin{lem}\label{L;Lemma2.8}\cite[Lemma 3.3.1 (iii)]{Z3}
Assume that $\mathbb{G}$ is a closed subset of $\mathbb{H}$. Assume that $\mathbb{F}$ is a normal closed subset of $\mathbb{G}$. Then $\mathbb{F}$ is a strongly normal closed subset of $\mathbb{G}$ if and only if $\mathrm{O}^\vartheta(\mathbb{G})\subseteq\mathbb{F}$.
\end{lem}
\begin{lem}\label{L;Lemma2.9}\cite[Lemma 4.4.2]{Z3}
Assume that $\mathbb{G}$ is a closed subset of $\mathbb{H}$. Then $$\mathrm{O}^\vartheta(\mathbb{G})=\langle\bigcup_{p\in\mathbb{G}}p^*p\rangle.$$
\end{lem}
For any a closed subset $\mathbb{G}$ of $\mathbb{H}$, call $\mathbb{G}$ a thin closed subset of $\mathbb{H}$ if
$\mathbb{G}=\mathrm{O}_\vartheta(\mathbb{G})$. Let $\gamma$ be the injective map from $\mathbb{H}$ to the power set of $\mathbb{H}$ that sends $p$ to $\{p\}$ for any $p\in \mathbb{H}$. By Lemmas \ref{L;Lemma2.1} and \ref{L;Lemma2.2}, let $\bullet$ be the associative binary operation on ${\mathrm{O}_\vartheta(\mathbb{H})}^\gamma$ that sends the elements $\{p\}, \{q\}$ to $pq$
for any $p, q\in\mathrm{O}_\vartheta(\mathbb{H})$. For any $\mathbb{G}\subseteq\mathbb{H}$, notice that $\mathbb{G}$ is a thin closed subset of $\mathbb{H}$ if and only if $\mathbb{G}^\gamma$ is a group with respect to $\bullet$ and the identity $\{e\}$. From now on, $\mathbb{G}^\gamma$ denotes the group with respect to $\bullet$ and the identity $\{e\}$ for any a thin closed subset $\mathbb{G}$ of $\mathbb{H}$. For any a prime $p$ and a finite thin closed subset $\mathbb{G}$ of $\mathbb{H}$, the $p$-rank of $\mathbb{G}$ is defined to be the $p$-rank of $\mathbb{G}^\gamma$. For any a prime $p$ and a finite closed subset $\mathbb{G}$ of $\mathbb{H}$, use $\mathrm{r}_p(\mathbb{G})$ to denote the largest $p$-rank of a thin closed subset of $\mathbb{H}$ that is contained in $\mathbb{G}$. The following lemmas are necessary:
\begin{lem}\label{L;Lemma2.10}
Assume that $\mathbb{G}\subseteq\mathbb{H}$. Then $\mathbb{G}^\gamma$ is an elementary abelian 2-group if and only if $\mathbb{G}$ is a thin closed subset of $\mathbb{H}$ and $\mathbb{G}=\mathrm{Sym}(\mathbb{G})$.
\end{lem}
\begin{proof}
If $p\!\in\!\mathbb{G}$, (H3) and (H2) imply that $p\in\mathrm{Sym}(\mathbb{G})\cap\mathrm{O}_\vartheta(\mathbb{G})$ if and only if $p^2=\{e\}$. The desired lemma thus follows from this statement and the above hypotheses.
\end{proof}
\begin{lem}\label{L;Lemma2.11}
Assume that $\mathbb{G}$ is a finite thin closed subset of $\mathbb{H}$ and $\mathbb{G}=\mathrm{Sym}(\mathbb{G})$. Then the dimension of $\mathbb{G}$ is equal to $\mathrm{r}_2(\mathbb{G})$ and $|\mathbb{G}|=2^{\mathrm{r}_2(\mathbb{G})}$.
\end{lem}
\begin{proof}
The desired lemma follows from the above hypotheses and Lemma \ref{L;Lemma2.10}.
\end{proof}
For any a closed subset $\mathbb{G}$ of $\mathbb{H}$, call $\mathbb{G}$ a residually thin closed subset of $\mathbb{H}$ if there are
$p\in\mathbb{N}$ and pairwise distinct closed subsets $\mathbb{F}_1, \mathbb{F}_2, \ldots, \mathbb{F}_p$ of $\mathbb{H}$
such that $\mathbb{F}_1=\{e\}$, $\mathbb{F}_p=\mathbb{G}$, and $\mathbb{F}_q$ is a strongly normal closed subset of $\mathbb{F}_{q+1}$ for any $q\in[1, p-1]$. Hence a thin closed subset of $\mathbb{H}$ is always a residually thin closed subset of $\mathbb{H}$. Call $\mathbb{H}$ a residually thin hypergroup if $\mathbb{H}$ itself is a residually thin closed subset of $\mathbb{H}$.

For any $\mathbb{F}, \mathbb{G}\subseteq\mathbb{H}$, set $[\mathbb{F}, \mathbb{G}]=\langle\{a: b\in\mathbb{F}, c\in\mathbb{G}, a\in b^*c^*bc\}\rangle$. For any $p\in\mathbb{N}\setminus\{1\}$ and $\mathbb{G}\!\subseteq\!\mathbb{H}$, set $\mathbb{G}^{(1)}=\mathbb{G}$ and $\mathbb{G}^{(p)}=[\mathbb{G}^{(p-1)}, \mathbb{G}]$ inductively. For any a closed subset $\mathbb{G}$ of $\mathbb{H}$, call $\mathbb{G}$ a nilpotent closed subset of $\mathbb{H}$ if there is $p\in\mathbb{N}$ such that $\mathbb{G}^{(p)}\!=\!\{e\}$. So a commutative thin closed subset of $\mathbb{H}$ is always a nilpotent closed subset of $\mathbb{H}$. Call $\mathbb{H}$ a nilpotent hypergroup if $\mathbb{H}$ itself is a nilpotent closed subset of $\mathbb{H}$ (see \cite{GZ}).

For any closed subsets $\mathbb{F}, \mathbb{G}$ of $\mathbb{H}$ and a map $\alpha$ from $\mathbb{F}$ to $\mathbb{G}$, call $\alpha$ a homomorphism from $\mathbb{F}$ to $\mathbb{G}$ if $e^\alpha=e$ and $(pq)^\alpha=p^\alpha q^\alpha$ for any $p, q\in \mathbb{F}$.
For any closed subsets $\mathbb{F}, \mathbb{G}$ of $\mathbb{H}$ and a homomorphism $\alpha$ from $\mathbb{F}$ to $\mathbb{G}$,
(H3) and (H2) imply that $(p^*)^\alpha\!=\!(p^\alpha)^*$ and $\mathrm{O}_\vartheta(\mathbb{F})^\alpha\subseteq\mathrm{O}_\vartheta(\mathbb{G})$ for any $p\in\mathbb{F}$.
For any closed subsets $\mathbb{F}, \mathbb{G}$ of $\mathbb{H}$, each bijective homomorphism from $\mathbb{F}$ to $\mathbb{G}$ is called an isomorphism from $\mathbb{F}$ to $\mathbb{G}$. For any closed subsets $\mathbb{F}$, $\mathbb{G}$ of $\mathbb{H}$, let $\mathrm{Iso}(\mathbb{F}, \mathbb{G})$ be the set of all isomorphisms from $\mathbb{F}$ to $\mathbb{G}$ and put
$\mathrm{Aut}(\mathbb{G})=\mathrm{Iso}(\mathbb{G}, \mathbb{G})$. For any closed subsets $\mathbb{F}, \mathbb{G}$ of $\mathbb{H}$, write $\mathbb{F}\simeq\mathbb{G}$ if $\mathrm{Iso}(\mathbb{F}, \mathbb{G})\neq \varnothing$.

For any a closed subset $\mathbb{G}$ of $\mathbb{H}$, $\mathrm{Aut}(\mathbb{G})$ contains the identity map $\epsilon_\mathbb{G}$.
For any closed subsets $\mathbb{F}, \mathbb{G}$ of $\mathbb{H}$, notice that $\alpha\!\in\!\mathrm{Iso}(\mathbb{F}, \mathbb{G})$ if and only if $\alpha^{-1}\!\in\!\mathrm{Iso}(\mathbb{G}, \mathbb{F})$. For any closed subsets $\mathbb{E}, \mathbb{F}, \mathbb{G}$ of $\mathbb{H}$, notice that $\alpha\in\mathrm{Iso}(\mathbb{E}, \mathbb{F})$ and $\beta\in\mathrm{Iso}(\mathbb{F}, \mathbb{G})$ imply that $\alpha\beta\in\mathrm{Iso}(\mathbb{E}, \mathbb{G})$. So $\simeq$ is an equivalence relation on the set of all closed subsets of $\mathbb{H}$. If $\mathbb{F}$ and $\mathbb{G}$ are closed subsets of $\mathbb{H}$, $\mathbb{F}$ is said to be isomorphic to $\mathbb{G}$ if $\mathbb{F}\simeq\mathbb{G}$.

For any a closed subset $\mathbb{G}$ of $\mathbb{H}$, $\mathrm{Aut}(\mathbb{G})$ is a group with respect to the composition of maps
and the identity $\epsilon_\mathbb{G}$. From now on, $\mathrm{Aut}(\mathbb{G})$ denotes the group with respect to the composition of maps and the identity $\epsilon_\mathbb{G}$ if $\mathbb{G}$ is a closed subset of $\mathbb{H}$. Call $\mathrm{Aut}(\mathbb{G})$ the automorphism group of $\mathbb{G}$ if $\mathbb{G}$ is a closed subset of $\mathbb{H}$. For any closed subsets $\mathbb{F}, \mathbb{G}$ of $\mathbb{H}$, notice that $\mathrm{Aut}(\mathbb{F})\cong\mathrm{Aut}(\mathbb{G})$ if $\mathbb{F}\simeq\mathbb{G}$. The automorphism group of a group $\mathbb{G}$ is $\mathrm{Aut}(\mathbb{G})$ if there is no confusion. The following lemma is necessary:
\begin{lem}\label{L;Lemma2.12}
Assume that $\mathbb{G}$ is a thin closed subset of $\mathbb{H}$. Then $\mathrm{Aut}(\mathbb{G})\!\cong\!\mathrm{Aut}(\mathbb{G}^\gamma)$.
\end{lem}
\begin{proof}
If $\alpha\in\mathrm{Aut}(\mathbb{G})$, let $\alpha_\gamma$ be the group automorphism in $\mathrm{Aut}(\mathbb{G}^\gamma)$ that sends
$\{p\}$ to $\{p^\alpha\}$ for any $p\in\mathbb{G}$. Notice that the map that sends $\alpha$ to $\alpha_\gamma$ for any $\alpha\in\mathrm{Aut}(\mathbb{G})$ is a group isomorphism from $\mathrm{Aut}(\mathbb{G})$ to $\mathrm{Aut}(\mathbb{G}^\gamma)$.
The desired lemma thus follows.
\end{proof}
\subsection{Elementary abelian 2-hypergroups}
Assume that $p\in\mathbb{N}$ and $\mathbb{H}$ has normal closed subsets $\mathbb{G}_1, \mathbb{G}_2, \ldots, \mathbb{G}_p$.
Recall that $\mathbb{G}_q\mathbb{G}_r=\mathbb{G}_r\mathbb{G}_q$ for any $q, r\in [1,p]$. Define
\[\widehat{\mathbb{G}_q}=\begin{cases}
\{e\}, &\ \text{if $p=1$},\\
\prod_{r\in [1,p]\setminus\{q\}}\mathbb{G}_r, &\ \text{if $p>1$}
\end{cases}\]
for any $q\in [1, p]$. If $\mathbb{H}=\mathbb{G}_q\widehat{\mathbb{G}_q}$ and $\mathbb{G}_q\cap\widehat{\mathbb{G}_q}=\{e\}$ for any $q\in [1, p]$, call $\mathbb{H}$ the direct product of its closed subsets $\mathbb{G}_1, \mathbb{G}_2, \ldots, \mathbb{G}_p$. The following lemmas are necessary:
\begin{lem}\label{L;Lemma2.13}\cite[Lemma 3.1.9]{Z3}
Assume that $p\!\in\!\mathbb{N}$ and $\mathbb{H}$ is the direct product of its closed subsets $\mathbb{G}_1, \mathbb{G}_2, \ldots, \mathbb{G}_p$. Then $qr\!=\!rq$ for any distinct $s, t\!\in\![1, p], q\!\in\!\mathbb{G}_s, r\!\in\!\mathbb{G}_t$.
\end{lem}
\begin{lem}\label{L;Lemma2.14}
Assume that $p\in\mathbb{N}$ and $\mathbb{H}$ is the direct product of its closed subsets $\mathbb{G}_1, \mathbb{G}_2, \ldots, \mathbb{G}_p$. Assume that $q\in\mathbb{H}$. Then there are unique $r_1, r_2, \ldots, r_p\in\mathbb{H}$ such that $q\in r_1r_2\cdots r_p$ and $r_s\in\mathbb{G}_s$ for any $s\in [1, p]$.
\end{lem}
\begin{proof}
The desired lemma follows from combining Lemmas \ref{L;Lemma2.3}, \ref{L;Lemma2.4}, and \ref{L;Lemma2.13}.
\end{proof}
Assume that $\mathbb{H}$ is the direct product of its closed subsets $\mathbb{G}_1, \mathbb{G}_2, \ldots, \mathbb{G}_p$. For any $q\in\mathbb{H}$, call the unique set $\{r_1, r_2, \ldots, r_p\}\setminus\{e\}$ occurred in
Lemma \ref{L;Lemma2.14} the support of $q$. For any $q\in\mathbb{H}$, notice that $q=e$ if and only if the support of $q$ is the empty set. For any $q\in\mathbb{H}$,
the number of all thick elements of $\mathbb{H}$ in the support of $q$ is denoted by $\mathrm{s}(q)$.
For any a closed subset $\mathbb{G}$ of $\mathbb{H}$, let $\mathrm{s}(\mathbb{G})$ be the number $\max\{\mathrm{s}(a): a\in\mathbb{G}\}$.

For any $q\in\mathbb{H}$, call $q$ an externally thick element of $\mathbb{H}$ if $q=e$ or all elements in the support of $q$ are the thick elements of $\mathbb{H}$. For any $q\in\mathbb{H}$, call $q$ an externally thin element of $\mathbb{H}$ if $q=e$ or all elements in the support of $q$ are the thin elements of $\mathbb{H}$. If $\mathbb{F}\subseteq\mathbb{H}$, let $\mathrm{Thick}(\mathbb{F})$ be the set of all externally thick elements of $\mathbb{H}$ in $\mathbb{F}$. If $\mathbb{F}\subseteq\mathbb{H}$, let $\mathrm{Thin}(\mathbb{F})$ be the set of all externally thin elements of $\mathbb{H}$ in $\mathbb{F}$. Notice that
$\mathbb{H}\!=\!\mathrm{Thick}(\mathbb{H})\mathrm{Thin}(\mathbb{H})$ by Lemmas \ref{L;Lemma2.13} and \ref{L;Lemma2.14}.
If $q\!\in\![1, p]$, Lemmas \ref{L;Lemma2.1} and \ref{L;Lemma2.2} give $\mathrm{Thin}(\mathbb{H})\!\subseteq\!\mathrm{O}_\vartheta(\mathbb{H})$ and $\mathrm{Thin}(\mathbb{G}_q)\!=\!\mathrm{O}_\vartheta(\mathbb{G}_q)$. The following lemmas are necessary:
\begin{lem}\label{L;Lemma2.15}
Assume that $p\in\mathbb{N}$ and $\mathbb{H}$ is the direct product of its closed subsets $\mathbb{G}_1, \mathbb{G}_2, \ldots, \mathbb{G}_p$. Assume that $\{e\}\subseteq\mathbb{F}\subseteq\mathbb{H}$.  Then $\mathrm{Thick}(\mathbb{F})\cap\mathrm{Thin}(\mathbb{F})=\{e\}$.
\end{lem}
\begin{proof}
The desired lemma follows from the above hypotheses and Lemma \ref{L;Lemma2.14}.
\end{proof}
\begin{lem}\label{L;Lemma2.16}
Assume that $p\in\mathbb{N}$ and $\mathbb{H}$ is the direct product of its closed subsets $\mathbb{G}_1, \mathbb{G}_2, \ldots, \mathbb{G}_p$. Then $\mathrm{O}_\vartheta(\mathbb{G}_1), \mathrm{O}_\vartheta(\mathbb{G}_2), \ldots, \mathrm{O}_\vartheta(\mathbb{G}_p)$ are thin closed subsets of $\mathbb{H}$ if and only if $\mathrm{Thin}(\mathbb{F})$ is a thin closed subset of $\mathbb{H}$ for any a closed subset $\F$ of $\mathbb{H}$.
\end{lem}
\begin{proof}
The desired lemma follows from combining Lemmas \ref{L;Lemma2.1}, \ref{L;Lemma2.2}, and \ref{L;Lemma2.13}.
\end{proof}
For any $q\in\mathbb{H}$, call the unique subset $r_1r_2\cdots r_p$ of $\mathbb{H}$ occurred in Lemma \ref{L;Lemma2.14} the cover of $q$. For any $q\in\mathbb{H}$, Lemma \ref{L;Lemma2.13} and (H2) thus imply that the cover of $q$ and the support of $q$ can be mutually determined. For any $q\in\mathbb{H}$, call $q$ a constrained element of $\mathbb{H}$ if the cover of $q$ equals precisely $\{q\}$.
The following lemma is necessary:
\begin{lem}\label{L;Lemma2.17}
Assume that $p\in\mathbb{N}$ and $\mathbb{H}$ is the direct product of its closed subsets $\mathbb{G}_1, \mathbb{G}_2, \ldots, \mathbb{G}_p$. Assume that $\mathbb{F}\subseteq\mathbb{H}$ and all elements in $\mathrm{O}_\vartheta(\mathbb{F})$ are the constrained elements of $\mathbb{H}$. Then $\mathrm{Thin}(\mathbb{F})=\mathrm{O}_\vartheta(\mathbb{F})$.
\end{lem}
\begin{proof}
The desired lemma follows from the above hypotheses and Lemma \ref{L;Lemma2.14}.
\end{proof}
For any $q, r\in\mathbb{H}$, set $q\preceq r$ if the support of $q$ is a subset of the support of $r$. So $\preceq$ is a preorder on $\mathbb{H}$. So $\preceq$ is a partial order on $\mathbb{H}$ if and only if all elements in $\mathbb{H}$ are the constrained elements of $\mathbb{H}$. If all elements in $\mathbb{H}$ are the constrained elements of $\mathbb{H}$ and $q\!\in\!\mathbb{H}$, Lemmas \ref{L;Lemma2.13} and \ref{L;Lemma2.14} give $q^+q^-=\{q\}$ for unique $q^+\!\in\!\mathrm{Thick}(\mathbb{H})$ and $q^-\!\in\!\mathrm{Thin}(\mathbb{H})$. If all elements in $\mathbb{H}$ are the constrained elements of $\mathbb{H}$, then $\mathbb{H}$ is called the constrained direct product of its closed subsets $\mathbb{G}_1, \mathbb{G}_2, \ldots, \mathbb{G}_p$. Moreover, call $\mathbb{H}$ an elementary abelian $2$-hypergroup if the following conditions hold together:
\begin{enumerate}[(E1)]
\item $\mathbb{H}$ is the constrained direct product of its closed subsets $\langle q_1\rangle, \langle q_2\rangle,\ldots, \langle q_p\rangle$;
\item These elements $q_1, q_2, \ldots, q_p$ of $\mathbb{H}$ occurred in (E1) are all involutions of $\mathbb{H}$.
\end{enumerate}

From now on, $\mathbb{H}$ denotes a fixed elementary abelian $2$-hypergroup, where $q_r$ is a fixed involution of $\mathbb{H}$ for any $r\in [1, p]$. Let $p^\sharp$ be the number of all thick elements of $\mathbb{H}$ contained in $\{q_1, q_2,\ldots, q_p\}$. According to (E1) and (E2), $\{q_1, q_2, \ldots, q_p\}$ is also a constrained subset of $\mathbb{H}$ whose definition is from \cite{Z3}. Notice that $\mathbb{H}$ is a commutative hypergroup by (E2) and Lemma \ref{L;Lemma2.13}. Hence (E1) and (E2) imply that $\mathbb{H}=\mathrm{Sym}(\mathbb{H})$. Furthermore, the combination of (E1), (E2), and Lemma \ref{L;Lemma2.14} implies that $|\mathbb{H}|=2^p$.

We are now ready to finish this section by simplifying the following presentation.

We shall quote the fact that $\preceq$ is a partial order on $\mathbb{H}$ without citation. We shall quote the fact that
$\mathbb{H}$ is a commutative hypergroup without citation. We shall quote the facts that $\mathbb{H}=\mathrm{Sym}(\mathbb{H})$
and $\mathbb{H}$ is a hypergroup of $2^p$ elements without citation.
\section{Closed subsets and strongly normal closed subsets}
In this present section, we determine all closed subsets and all strongly normal closed subsets of $\mathbb{H}$. We also give the numbers of all closed subsets and all strongly normal closed subsets of $\mathbb{H}$. As a preparation, we first display the following lemmas:
\begin{lem}\label{L;Lemma3.1}
Assume that $r\!\in\!\mathbb{H}$. Then $r^{2s+2}\!=\!\{a: a\in\mathbb{H}, a\preceq r^+\}$ for any $s\in\mathbb{N}_0$. Moreover, if $r\in\mathrm{Thick}(\mathbb{H})$, then $r=r^+$, $\langle r\rangle=\{a: a\in\mathbb{H}, a\preceq r\}$, and $|\langle r\rangle|=2^{\mathrm{s}(r)}$.
\end{lem}
\begin{proof}
The first statement follows from (E1) and (E2). For the second statement, notice that $r=r^+$ by (H2). So $r^2=r^4\!=\!\{a: a\in\mathbb{H}, a\preceq r\}$ by the first statement. Therefore $\{r\}\subseteq\{a: a\in\mathbb{H}, a\preceq r\}\subseteq\langle r\rangle$. As $r^2r^2=r^4=r^2$ by (H1), notice that $\{a: a\in\mathbb{H}, a\preceq r\}$ is a closed subset of $\mathbb{H}$. So $\langle r\rangle\subseteq\{a: a\in\mathbb{H}, a\preceq r\}$. The desired lemma thus follows from combining the above discussion, (E1), and (E2).
\end{proof}
\begin{lem}\label{L;Lemma3.2}
Assume that $r, s\in\mathrm{Thick}(\mathbb{H})$. Then $r=s$ if and only if $\langle r\rangle=\langle s\rangle$.
\end{lem}
\begin{proof}
The desired lemma follows from the above hypotheses and Lemma \ref{L;Lemma3.1}.
\end{proof}
\begin{lem}\label{L;Lemma3.3}
Assume that $\mathbb{G}$ is a closed subset of $\mathbb{H}$ and $r\!\in\!\mathbb{G}$. Then $r^+\!\in\!\mathrm{Thick}(\mathbb{G})$, $r^-\in\mathrm{Thin}(\mathbb{G})$, and $\mathbb{G}=\mathrm{Thick}(\mathbb{G})\mathrm{Thin}(\mathbb{G})$.
\end{lem}
\begin{proof}
Notice that $\mathrm{Thick}(\mathbb{G})\mathrm{Thin}(\mathbb{G})\!\subseteq\!\mathbb{G}$. By Lemma \ref{L;Lemma3.1}, notice that
$r^+\!\in\!\mathrm{Thick}(\mathbb{H})$, $r^-\in\mathrm{Thin}(\mathbb{H})$, $r^+r^-=\{r\}$, and $r^2=\{a: a\preceq r^+\}$. This implies that
$r^+\in r^2$ and $r^+\in\mathrm{Thick}(\mathbb{G})$. Notice that $r^-\in\mathrm{Thin}(\mathbb{G})$ since $r^-\in r^+r$ by (H3). As $r$ is chosen from $\mathbb{G}$ arbitrarily, the desired lemma thus follows from the above discussion.
\end{proof}
\begin{lem}\label{L;Lemma3.4}
Assume that $\mathbb{G}$ is a closed subset of $\mathbb{H}$. Then $\mathrm{Thick}(\mathbb{G})$ and $\mathrm{Thin}(\mathbb{G})$ are closed subsets of $\mathbb{H}$. Moreover, $|\mathbb{G}|\!=\!|\mathrm{Thick}(\mathbb{G})\mathrm{Thin}(\mathbb{G})|\!=\!|\mathrm{Thick}(\mathbb{G})||\mathrm{Thin}(\mathbb{G})|$.
\end{lem}
\begin{proof}
If $r\in[1, p]$, notice that $\mathrm{O}_\vartheta(\langle q_r\rangle)$ is a closed subset of $\mathbb{H}$. The first statement follows from combining Lemmas \ref{L;Lemma2.15}, \ref{L;Lemma2.16}, (E1), and (E2). The desired lemma thus follows from combining the first statement, Lemmas \ref{L;Lemma3.3}, \ref{L;Lemma2.15}, \ref{L;Lemma2.4}, and (E1).
\end{proof}
\begin{lem}\label{L;Lemma3.5}
Assume that $\mathbb{G}$ is a closed subset of $\mathbb{H}$. Then there is a unique $r\in\mathbb{G}$ such that $\mathrm{Thick}(\mathbb{G})=\langle r\rangle$ and $\mathrm{s}(\mathbb{G})=\mathrm{s}(r)$.
\end{lem}
\begin{proof}
Pick $r\in\mathbb{G}$. Then $r^+\in \mathrm{Thick}(\mathbb{G})$ and $\mathrm{s}(r)=\mathrm{s}(r^+)$ by Lemmas \ref{L;Lemma3.3} and \ref{L;Lemma3.1}. There is no loss to assume further that $r\in\mathrm{Thick}(\mathbb{G})$ and $\mathrm{s}(\mathbb{G})=\mathrm{s}(r)$. Lemma \ref{L;Lemma3.4} implies that $\langle r\rangle\subseteq\mathrm{Thick}(\mathbb{G})$. Assume that $s\in\mathrm{Thick}(\mathbb{G})\setminus\langle r\rangle$. By Lemma \ref{L;Lemma3.1}, there is $t\in s^2\setminus\{e\}$ such
that the intersection of the supports of $t$ and $r$ is the empty set. Set $rt=\{u\}$ by (E1). Hence $u\in\mathrm{Thick}(\mathbb{G})$ and $\mathrm{s}(\mathbb{G})=\mathrm{s}(r)<\mathrm{s}(r)+\mathrm{s}(t)=\mathrm{s}(u)$ by Lemma \ref{L;Lemma3.4}. This is absurd. The desired lemma thus follows from Lemma \ref{L;Lemma3.2}.
\end{proof}
\begin{lem}\label{L;Lemma3.6}
Assume that $\mathbb{G}$ is a closed subset of $\mathbb{H}$. Then $\mathrm{Thick}(\mathbb{G})=\mathrm{O}^\vartheta(\mathbb{G})$ and $|\mathrm{Thick}(\mathbb{G})|=|\mathrm{O}^\vartheta(\mathbb{G})|=2^{\mathrm{s}(\mathbb{G})}$.
\end{lem}
\begin{proof}
Notice that $\mathrm{O}^\vartheta(\mathbb{G})\!\subseteq\!\mathrm{Thick}(\mathbb{G})$ by combining Lemmas \ref{L;Lemma3.1}, \ref{L;Lemma3.4}, and \ref{L;Lemma2.9}. By Lemma \ref{L;Lemma3.5}, there is $r\!\in\!\mathrm{Thick}(\mathbb{G})$ such that $\mathrm{Thick}(\mathbb{G})\!=\!\langle r\rangle$. As $r\in r^2$ by Lemma \ref{L;Lemma3.1}, Lemma \ref{L;Lemma2.9} implies that
$\mathrm{Thick}(\mathbb{G})\!\subseteq\!\mathrm{O}^\vartheta(\mathbb{G})$. So $\mathrm{Thick}(\mathbb{G})\!=\!\mathrm{O}^\vartheta(\mathbb{G})$. The desired lemma thus follows from combining the above discussion, Lemmas \ref{L;Lemma3.5}, and \ref{L;Lemma3.1}.
\end{proof}
\begin{lem}\label{L;Lemma3.7}
Assume that $\mathbb{G}$ is a closed subset of $\mathbb{H}$. Then $\mathrm{Thin}(\mathbb{G})\!=\!\mathrm{O}_\vartheta(\mathbb{G})$ and
$\mathrm{O}_\vartheta(\mathbb{G})^\gamma$ is an elementary abelian 2-group of 2-rank $\mathrm{r}_2(\mathbb{G})$. Moreover, the dimension of $\mathrm{O}_\vartheta(\mathbb{G})$ is equal to $\mathrm{r}_2(\mathbb{G})$ and $|\mathrm{Thin}(\mathbb{G})|=|\mathrm{O}_\vartheta(\mathbb{G})|=2^{\mathrm{r}_2(\mathbb{G})}$.
In particular, $\mathrm{O}_\vartheta(\mathbb{H})^\gamma$ is an elementary abelian 2-group of 2-rank $\mathrm{r}_2(\mathbb{H})$ and
$|\mathrm{Thin}(\mathbb{H})|=|\mathrm{O}_\vartheta(\mathbb{H})|=2^{\mathrm{r}_2(\mathbb{H})}$.
\end{lem}
\begin{proof}
The first statement is from combining (E1), Lemmas \ref{L;Lemma2.17}, \ref{L;Lemma3.4}, \ref{L;Lemma2.10}. The desired lemma follows from combining the first statement, Lemmas \ref{L;Lemma3.4}, and \ref{L;Lemma2.11}.
\end{proof}
\begin{lem}\label{L;Lemma3.8}
Assume that $\mathbb{G}$ is a closed subset of $\mathbb{H}$. Then $\mathbb{G}\!=\!\mathrm{O}^\vartheta(\mathbb{G})\mathrm{O}_\vartheta(\mathbb{G})$ and $\mathrm{O}^\vartheta(\mathbb{G})\cap\mathrm{O}_\vartheta(\mathbb{G})\!=\!\{e\}$. Moreover, $|\mathbb{G}|\!=\!|\mathrm{O}^\vartheta(\mathbb{G})\mathrm{O}_\vartheta(\mathbb{G})|
\!=\!|\mathrm{O}^\vartheta(\mathbb{G})||\mathrm{O}_\vartheta(\mathbb{G})|\!=\!2^{\mathrm{s}(\mathbb{G})+\mathrm{r}_2(\mathbb{G})}$.
In particular, $\mathbb{H}=\mathrm{O}^\vartheta(\mathbb{H})\mathrm{O}_\vartheta(\mathbb{H})$, $\mathrm{O}^\vartheta(\mathbb{H})\cap\mathrm{O}_\vartheta(\mathbb{H})=\{e\}$, $\mathrm{s}(\mathbb{H})=p^\sharp$, $\mathrm{r}_2(\mathbb{H})=p-p^\sharp$.
\end{lem}
\begin{proof}
The first statement is from combining Lemmas \ref{L;Lemma2.15}, \ref{L;Lemma3.3}, \ref{L;Lemma3.6}, \ref{L;Lemma3.7}.
The desired lemma follows from combining the first statement, Lemmas \ref{L;Lemma3.4}, \ref{L;Lemma3.6}, and \ref{L;Lemma3.7}.
\end{proof}
\begin{lem}\label{L;Lemma3.9}
Assume that $\mathbb{G}$ is a closed subset of $\mathbb{H}$ and $r\!\in\!\mathrm{Thick}(\mathbb{G})$. Assume that
$\mathbb{F}\subseteq\mathbb{G}$ and $\mathbb{F}^\gamma$ is a subgroup of the elementary abelian 2-group $\mathrm{O}_\vartheta(\mathbb{G})^\gamma$. Then $\langle r\rangle\mathbb{F}$ is a closed subset of $\mathbb{H}$, where
$\langle r\rangle\mathbb{F}\subseteq\mathbb{G}$, $\mathrm{O}^\vartheta(\langle r\rangle\mathbb{F})=\langle r\rangle$, and
$\mathrm{O}_\vartheta(\langle r\rangle\mathbb{F})=\mathbb{F}$.
\end{lem}
\begin{proof}
Lemma \ref{L;Lemma2.10} shows that $\mathbb{F}$ is a thin closed subset of $\mathbb{H}$. Lemma \ref{L;Lemma2.3} shows that
$\langle r\rangle\mathbb{F}$ is a closed subset of $\mathbb{H}$ and $\langle r\rangle\mathbb{F}\subseteq\mathbb{G}$. Hence $\langle r\rangle\subseteq\mathrm{O}^\vartheta(\langle r\rangle\mathbb{F})$ and $\mathbb{F}\subseteq\mathrm{O}_\vartheta(\langle r\rangle\mathbb{F})$ by Lemmas \ref{L;Lemma3.4} and \ref{L;Lemma3.6}. Hence $|\langle r\rangle\mathbb{F}|\!\leq\!|\langle r\rangle||\mathbb{F}|\leq |\mathrm{O}^\vartheta(\langle r\rangle\mathbb{F})||\mathrm{O}_\vartheta(\langle r\rangle\mathbb{F})|\!=\!|\langle r\rangle\mathbb{F}|$ by combining Lemmas \ref{L;Lemma2.1}, \ref{L;Lemma2.2}, \ref{L;Lemma3.8}. The desired lemma follows from this discussion.
\end{proof}
\begin{lem}\label{L;Lemma3.10}
Assume that $\mathbb{G}$ is a closed subset of $\mathbb{H}$ and $r$ is the unique element in $\mathbb{G}$ that satisfies  the equalities $\mathrm{Thick}(\mathbb{G})=\langle r\rangle$ and $\mathrm{s}(\mathbb{G})=\mathrm{s}(r)$. Assume that $\mathbb{F}\!\subseteq\!\mathbb{G}$. Then $\mathbb{F}$ is a strongly normal closed subset of $\mathbb{G}$ if and only if
$\mathrm{O}^\vartheta(\mathbb{F})=\mathrm{O}^\vartheta(\mathbb{G})=\langle r\rangle$ and there exists $\mathbb{E}\subseteq\mathbb{G}$ such that $\mathbb{E}^\gamma$ is a subgroup of the elementary abelian 2-group $\mathrm{O}_\vartheta(\mathbb{G})^\gamma$, $\mathbb{F}=\langle r\rangle\mathbb{E}$, $\langle r\rangle\cap \mathbb{E}=\{e\}$, $\mathrm{O}_\vartheta(\mathbb{F})=\mathbb{E}$.
\end{lem}
\begin{proof}
For one direction, the combination of Lemmas \ref{L;Lemma2.8}, \ref{L;Lemma2.9}, \ref{L;Lemma3.5}, \ref{L;Lemma3.6} implies that $\mathrm{O}^\vartheta(\mathbb{F})\subseteq\mathrm{O}^\vartheta(\mathbb{G})=\mathrm{Thick}(\mathbb{G})=\langle r\rangle\subseteq\mathrm{Thick}(\mathbb{F})=\mathrm{O}^\vartheta(\mathbb{F})$. Hence the combination of Lemmas \ref{L;Lemma3.8}, \ref{L;Lemma3.5}, and \ref{L;Lemma3.6} implies that $\mathbb{F}=\langle r\rangle\mathrm{O}_\vartheta(\mathbb{F})$ and $\langle r\rangle\cap\mathrm{O}_\vartheta(\mathbb{F})=\{e\}$. The desired lemma thus follows from combining Lemmas \ref{L;Lemma3.4},  \ref{L;Lemma3.7}, \ref{L;Lemma2.10}, \ref{L;Lemma3.9}, and \ref{L;Lemma2.8}.
\end{proof}
\begin{lem}\label{L;Lemma3.11}
Assume that $\mathbb{G}$ is a closed subset of $\mathbb{H}$. Then the number of all closed subsets of $\mathbb{H}$ contained in $\mathbb{G}$ is equal to
$$2^{\mathrm{s}(\mathbb{G})}\sum_{r=0}^{\mathrm{r}_2(\mathbb{G})}\binom{\mathrm{r}_2(\mathbb{G})}{r}_2.$$
On the other hand, the number of all strongly normal closed subsets of $\mathbb{G}$ is equal to
$$\sum_{r=0}^{\mathrm{r}_2(\mathbb{G})}\binom{\mathrm{r}_2(\mathbb{G})}{r}_2.$$
\end{lem}
\begin{proof}
If $\mathbb{E}, \mathbb{F}\subseteq\mathbb{G}$ and $\mathbb{E}, \mathbb{F}$ are closed subsets of $\mathbb{H}$, Lemma \ref{L;Lemma3.8}
implies that $\mathbb{E}=\mathbb{F}$ if and only if $\mathrm{O}^\vartheta(\mathbb{E})=\mathrm{O}^\vartheta(\mathbb{F})$ and
$\mathrm{O}_\vartheta(\mathbb{E})=\mathrm{O}_\vartheta(\mathbb{F})$. The first statement follows from combining Lemmas \ref{L;Lemma3.5}, \ref{L;Lemma3.6}, \ref{L;Lemma3.2}, \ref{L;Lemma3.7}, \ref{L;Lemma3.9}, and a direct computation. So the desired lemma follows from combining Lemmas \ref{L;Lemma3.7}, \ref{L;Lemma3.9}, \ref{L;Lemma3.10}, and a direct computation.
\end{proof}
We are now ready to list the main results of this section as the following theorems:
\begin{thm}\label{T;TheoremA}
Assume that $\mathbb{G}\subseteq\mathbb{H}$. Then $\mathbb{G}$ is a closed subset of $\mathbb{H}$ if and only if there exist $r\in\mathrm{Thick}(\mathbb{H})$ and $\mathbb{F}\subseteq\mathbb{H}$ such that $\mathrm{O}^\vartheta(\mathbb{G})=\langle r\rangle$,
$\mathbb{F}^\gamma$ is a subgroup of the elementary abelian 2-group $\mathrm{O}_\vartheta(\mathbb{H})^\gamma$, $\mathbb{G}=\langle r\rangle\mathbb{F}$, $\langle r\rangle\cap\mathbb{F}=\{e\}$, and $\mathrm{O}_\vartheta(\mathbb{G})=\mathbb{F}$.
\end{thm}
\begin{proof}
The desired theorem follows from combining Lemmas \ref{L;Lemma3.3}, \ref{L;Lemma3.5}, \ref{L;Lemma3.7}, \ref{L;Lemma3.9}.
\end{proof}
\begin{thm}\label{T;TheoremB}
Assume that $\mathbb{G}\subseteq\mathbb{H}$ and $r$ denotes the unique element in $\mathbb{H}$ whose support contains precisely all thick elements of $\mathbb{H}$ contained in $\{q_1, q_2, \ldots, q_p\}$. Then $\mathbb{G}$ is a strongly normal closed subset of $\mathbb{H}$ if and only if $\mathrm{O}^\vartheta(\mathbb{G})=\mathrm{O}^\vartheta(\mathbb{H})=\langle r\rangle$ and there exists $\mathbb{F}\subseteq\mathbb{H}$ such that $\mathbb{F}^\gamma$ is a subgroup of the elementary abelian 2-group $\mathrm{O}_\vartheta(\mathbb{H})^\gamma$, $\mathbb{G}=\langle r\rangle\mathbb{F}$, $\langle r\rangle\cap\mathbb{F}=\{e\}$,
$\mathrm{O}_\vartheta(\mathbb{G})=\mathbb{F}$.
\end{thm}
\begin{proof}
The desired theorem follows from a direct computation and Lemma \ref{L;Lemma3.10}.
\end{proof}
\begin{thm}\label{T;TheoremC}
The number of all closed subsets of $\mathbb{H}$ is equal to $$2^{p^\sharp}\sum_{r=0}^{p-p^\sharp}\binom{p-p^\sharp}{r}_2.$$
On the other hand, the number of all strongly normal closed subsets of $\mathbb{H}$ is equal to
$$\sum_{r=0}^{p-p^\sharp}\binom{p-p^\sharp}{r}_2.$$
\end{thm}
\begin{proof}
The desired theorem follows from an application of Lemmas \ref{L;Lemma3.11} and \ref{L;Lemma3.8}.
\end{proof}
For a corollary of Theorem \ref{T;TheoremA}, it is necessary to introduce the following lemmas:
\begin{lem}\label{L;Lemma3.15}
Assume that $\mathbb{G}$ is a closed subset of $\mathbb{H}$ and $r$ is the unique element in $\mathbb{G}$ that satisfies the
equalities $\mathrm{Thick}(\mathbb{G})\!=\!\langle r\rangle$ and $\mathrm{s}(\mathbb{G})\!=\!\mathrm{s}(r)$. Assume that $\mathrm{r}_2(\mathbb{G})\!\in\!\mathbb{N}$. Assume that $\mathbb{F}\subseteq\mathbb{G}$ and $\mathbb{F}^\gamma$ is a subgroup of the elementary abelian 2-group $\mathrm{O}_\vartheta(\mathbb{G})^\gamma$. If the 2-rank of $\mathbb{F}^\gamma$ is equal to $\mathrm{r}_2(\mathbb{G})-1$, then $\langle r\rangle\mathbb{F}$ is a maximal closed subset of $\mathbb{G}$.
\end{lem}
\begin{proof}
The desired lemma follows from an application of Lemmas \ref{L;Lemma3.9} and \ref{L;Lemma3.8}.
\end{proof}
\begin{lem}\label{L;Lemma3.16}
Assume that $\mathbb{G}$ is a closed subset of $\mathbb{H}$ and $\mathbb{F}$ is the Frattini closed subset of $\mathbb{G}$.
Then $\mathbb{F}\cap\mathrm{O}_\vartheta(\mathbb{G})=\{e\}$.
\end{lem}
\begin{proof}
The desired lemma follows from combining Lemmas \ref{L;Lemma3.15}, \ref{L;Lemma3.9}, and \ref{L;Lemma3.7}.
\end{proof}
\begin{cor}\label{C;Corollary3.17}
Assume that $\mathbb{G}$ is a closed subset of $\mathbb{H}$ and $\mathbb{F}$ is the Frattini closed subset of $\mathbb{G}$.
Then $\mathbb{F}=\{e\}$. In particular, the Frattini closed subset of $\mathbb{H}$ equals $\{e\}$.
\end{cor}
\begin{proof}
Theorem \ref{T;TheoremA} shows that $\mathbb{G}=\langle r\rangle\mathrm{O}_\vartheta(\mathbb{G})$ for some $r\in\mathrm{Thick}(\mathbb{G})$. By Lemma \ref{L;Lemma3.16} and (H2), there is no loss to require that $\mathrm{s}(r)\in\mathbb{N}$ and $\{q_1, q_2,\ldots, q_{\mathrm{s}(r)}\}$ is the support of $r$. So $\langle r\rangle=\langle q_1, q_2,\ldots, q_{\mathrm{s}(r)}\rangle$ by Lemma \ref{L;Lemma3.1}. Assume that $\mathbb{F}\setminus\mathrm{O}_\vartheta(\mathbb{G})\neq\varnothing$.

According to Lemma \ref{L;Lemma3.1}, there is also no loss to require that $q_1\in\mathbb{F}$. If $\mathrm{s}(r)=1$, notice that $q_1\in\mathrm{O}_\vartheta(\mathbb{G})$ by Lemma \ref{L;Lemma2.6}. This is a contradiction. Assume further that $\mathrm{s}(r)\!\in\!\mathbb{N}\setminus\{1\}$. Then $q_1\!\in\!\langle q_2, q_3,\ldots, q_{\mathrm{s}(r)}\rangle\mathrm{O}_\vartheta(\mathbb{G})$ by Lemma \ref{L;Lemma2.6}. Lemmas \ref{L;Lemma2.7} and \ref{L;Lemma2.14} also imply that $q_1\not\in\langle q_2, q_3,\ldots, q_{\mathrm{s}(r)}\rangle\mathrm{O}_\vartheta(\mathbb{G})$.
This is also a contradiction. The desired corollary thus follows from the above contradictions and Lemma \ref{L;Lemma3.16}.
\end{proof}
For another corollary of Theorem \ref{T;TheoremA}, it is necessary to give the following lemmas:
\begin{lem}\label{L;Lemma3.18}
Assume that $\mathbb{G}\!=\!\langle r_1, r_2,\ldots, r_s\rangle$ for some $s\in\mathbb{N}$ and $r_1, r_2,\ldots, r_s\!\in\!\mathbb{H}$. Then $\mathbb{G}=\langle r_1^+, r_2^+,\ldots, r_s^+\rangle\langle r_1^-, r_2^-,\ldots, r_s^-\rangle$. Moreover, $\mathrm{O}^\vartheta(\mathbb{G})\!=\!\langle r_1^+, r_2^+,\ldots, r_s^+\rangle$ and
$\mathrm{O}_\vartheta(\mathbb{G})=\langle r_1^-, r_2^-,\ldots, r_s^-\rangle$.
\end{lem}
\begin{proof}
The first statement follows from Lemma \ref{L;Lemma3.3}. By combining Lemmas \ref{L;Lemma2.9}, \ref{L;Lemma3.1}, \ref{L;Lemma3.4}, \ref{L;Lemma3.7}, notice that $\langle r_1^+, r_2^+,\ldots, r_s^+\rangle\subseteq\mathrm{O}^\vartheta(\mathbb{G})$ and
$\langle r_1^-, r_2^-,\ldots, r_s^-\rangle\subseteq\mathrm{O}_\vartheta(\mathbb{G})$. Notice that
$|\mathbb{G}|\!\leq\!|\langle r_1^+, r_2^+,\ldots, r_s^+\rangle||\langle r_1^-, r_2^-,\ldots, r_s^-\rangle|\!\leq\!|\mathrm{O}^\vartheta(\mathbb{G})||\mathrm{O}_\vartheta(\mathbb{G})|\!=\!|\mathbb{G}|$
by combining Lemmas \ref{L;Lemma2.1}, \ref{L;Lemma2.2}, \ref{L;Lemma3.8}. The desired lemma thus follows from the above discussion.
\end{proof}
\begin{lem}\label{L;Lemma3.19}
Assume that $\mathbb{G}$ is a closed subset of $\mathbb{H}$. Then the dimension of $\mathbb{G}$ is no less than $\mathrm{r}_2(\mathbb{G})$.
\end{lem}
\begin{proof}
The basis of $\{e\}$ is the empty set. So $\mathrm{r}_2(\{e\})=0$ and the dimension of $\mathbb{G}$ is zero if and only if $\mathbb{G}=\{e\}$. There is no loss to assume that $\mathbb{G}\neq\{e\}$. Assume that $s\in\mathbb{N}$ and $\{r_1, r_2,\ldots, r_s\}$ is a basis of $\mathbb{G}$. Assume that $\mathrm{r}_2(\mathbb{G})\in\mathbb{N}\setminus[1, s]$. Notice that $\mathrm{O}_\vartheta(\mathbb{G})\!=\!\langle r_1^-, r_2^-,\ldots, r_s^-\rangle$ by Lemma \ref{L;Lemma3.18}. By Lemma \ref{L;Lemma2.10}, the 2-rank of $\mathrm{O}_\vartheta(\mathbb{G})^\gamma$ is no more than $s$. This contradicts Lemma \ref{L;Lemma3.7}. The desired lemma thus follows.
\end{proof}
\begin{lem}\label{L;Lemma3.20}
Assume that $\mathbb{G}$ is a closed subset of $\mathbb{H}$ and $\mathrm{O}_\vartheta(\mathbb{G})\neq\{e\}$. Assume that $\{r_1, r_2, \ldots, r_s\}\subseteq\mathbb{G}$ for some $s\in\mathbb{N}$. Then $\{r_1, r_2, \ldots, r_s\}$ is a basis of $\mathbb{G}$ if and only if
$s=\mathrm{r}_2(\mathbb{G})$, the union of the supports of $r_1^+, r_2^+, \ldots, r_s^+$ is a set of $\mathrm{s}(\mathbb{G})$ elements,
and $\{r_1^-, r_2^-, \ldots, r_s^-\}$ is a basis of the thin closed subset $\mathrm{O}_\vartheta(\mathbb{G})$ of $\mathbb{H}$.
\end{lem}
\begin{proof}
As $\mathrm{O}_\vartheta(\mathbb{G})\neq\{e\}$, the combination of Lemmas \ref{L;Lemma3.7}, \ref{L;Lemma3.4}, \ref{L;Lemma2.11} shows that each basis of $\mathrm{O}_\vartheta(\mathbb{G})$ is not the empty set. For one direction, $\langle r_1^+, r_2^+, \ldots, r_s^+\rangle\subseteq\mathrm{Thick}(\mathbb{G})$ by Lemmas \ref{L;Lemma3.3} and \ref{L;Lemma3.4}.
As the union of the supports of $r_1^+, r_2^+, \ldots, r_s^+$ is a set of $\mathrm{s}(\mathbb{G})$ elements,
notice that $\mathrm{O}^\vartheta(\mathbb{G})\!=\!\langle r_1^+, r_2^+, \ldots, r_s^+\rangle$ by combining Lemmas \ref{L;Lemma3.6}, \ref{L;Lemma3.5}, \ref{L;Lemma2.7}. By combining Theorem \ref{T;TheoremA}, Lemmas \ref{L;Lemma3.7}, \ref{L;Lemma3.18}, \ref{L;Lemma3.19}, $\{r_1, r_2, \ldots, r_s\}$ is a basis of $\mathbb{G}$. For the other direction,
$s=\mathrm{r}_2(\mathbb{G})$ by the above discussion and Lemma \ref{L;Lemma3.19}. The desired lemma follows from combining Lemmas \ref{L;Lemma3.18}, \ref{L;Lemma3.6}, \ref{L;Lemma3.5}, \ref{L;Lemma2.7}, and \ref{L;Lemma3.7}.
\end{proof}
\begin{lem}\label{L;Lemma3.21}
Assume that $\mathbb{G}$ is a closed subset of $\mathbb{H}$ and $\mathrm{O}^\vartheta(\mathbb{G})\neq\mathrm{O}_\vartheta(\mathbb{G})=\{e\}$.
Then there is a unique $r\in \mathbb{G}$ such that $\mathrm{s}(r)=\mathrm{s}(\mathbb{G})$ and $\mathbb{G}$ has the unique basis $\{r\}$.
\end{lem}
\begin{proof}
The desired lemma follows from combining Lemmas \ref{L;Lemma3.8}, \ref{L;Lemma3.6}, \ref{L;Lemma3.5}, and \ref{L;Lemma3.2}.
\end{proof}
\begin{cor}\label{C;Corollary3.22}
Assume that $\mathbb{G}$ is a closed subset of $\mathbb{H}$. Then the dimension of $\mathbb{G}$ is equal to
\[\begin{cases}
1, &\ \text{if}\ \mathrm{O}^\vartheta(\mathbb{G})\neq\mathrm{O}_\vartheta(\mathbb{G})=\{e\},\\
\mathrm{r}_2(\mathbb{G}), &\ \text{otherwise}.
\end{cases}\]
As a particular case of the above displayed formula, the dimension of $\mathbb{H}$ is equal to
\[\begin{cases}
1, &\ \text{if}\ \mathrm{O}^\vartheta(\mathbb{H})\neq\mathrm{O}_\vartheta(\mathbb{H})=\{e\},\\
p-p^\sharp, &\ \text{otherwise}.
\end{cases}\]
\end{cor}
\begin{proof}
The first statement follows from combining Lemmas \ref{L;Lemma3.20}, \ref{L;Lemma3.21}, and Theorem \ref{T;TheoremA}. The desired corollary thus follows from the first statement and Lemma \ref{L;Lemma3.8}.
\end{proof}
We display the remaining applications of Theorem \ref{T;TheoremA} as the following corollaries:
\begin{cor}\label{C;Corollary3.23}
Assume that $\mathbb{G}$ is a closed subset of $\mathbb{H}$. Then $\mathbb{G}$ is a residually thin closed subset of $\mathbb{H}$
if and only if $\mathbb{G}=\mathrm{O}_\vartheta(\mathbb{G})$. In particular, $\mathbb{H}$ is a residually thin hypergroup if and only if $\mathbb{H}=\mathrm{O}_\vartheta(\mathbb{H})$.
\end{cor}
\begin{proof}
For one direction, assume that $\mathbb{G}$ is a residually thin closed subset of $\mathbb{H}$ and $\mathbb{G}\neq \mathrm{O}_\vartheta(\mathbb{G})$. So there are $r\in \mathbb{N}$ and pairwise distinct closed subsets $\mathbb{F}_1, \mathbb{F}_2,\ldots, \mathbb{F}_r$ of $\mathbb{H}$ such that $\mathbb{F}_1=\{e\}$, $\mathbb{F}_r=\mathbb{G}$, and $\mathbb{F}_s$ is a strongly normal closed subset of $\mathbb{F}_{s+1}$ for any $s\!\in\! [1, r\!-\!1]$. As $\mathbb{F}_1\!\!=\!\!\{e\}$, let $t$ be the largest subscript such that $\mathbb{F}_t$ is a thin closed subset of $\mathbb{H}$. So $t\!\in\![1, r\!-\!1]$ as $\mathbb{G}\!\neq\!\mathrm{O}_\vartheta(\mathbb{G})$. By combining the choice of $t$, Theorem \ref{T;TheoremA}, Lemma \ref{L;Lemma3.1}, there is no loss to let $q_1\!\in\!\mathrm{Thick}(\mathbb{F}_{t+1})$. So $\{q_1\}\subseteq q_1\mathbb{F}_tq_1\subseteq\mathbb{F}_t$
by (H2). This indeed contradicts the choice of $t$. The desired corollary thus follows.
\end{proof}
\begin{cor}\label{C;Corollary3.24}
Assume that $\mathbb{G}$ is a closed subset of $\mathbb{H}$. Then $\mathbb{G}^{(r)}=\mathrm{O}^\vartheta(\mathbb{G})$ for any $r\in\mathbb{N}\setminus\{1\}$. Moreover, $\mathbb{G}$ is a nilpotent closed subset of $\mathbb{H}$ if and only if $\mathbb{G}=\mathrm{O}_\vartheta(\mathbb{G})$. In particular, $\mathbb{H}$ is a nilpotent hypergroup if and only if $\mathbb{H}=\mathrm{O}_\vartheta(\mathbb{H})$.
\end{cor}
\begin{proof}
If $r\in\mathbb{N}\setminus\{1\}$, notice that $\mathbb{G}^{(r)}=[\mathbb{G}^{(r-1)}, \mathbb{G}]\subseteq [\mathbb{G}, \mathbb{G}]=\mathrm{O}^\vartheta(\mathbb{G})$ by combining Lemmas \ref{L;Lemma2.7}, \ref{L;Lemma2.9}, and (H2). If $r\!\in\!\mathbb{N}\setminus\{1\}$, notice that $\mathrm{O}^\vartheta(\mathbb{G})\subseteq[\mathbb{G}^{(r-1)}, \mathbb{G}]=\mathbb{G}^{(r)}$ by Lemma \ref{L;Lemma2.9} and (H2). Hence $\mathbb{G}^{(r)}=\mathrm{O}^\vartheta(\mathbb{G})$ for any
$r\in\mathbb{N}\setminus\{1\}$. The desired corollary thus follows from combining the first statement,
Theorem \ref{T;TheoremA}, and (H2).
\end{proof}
\begin{rem}\label{R;Remark3.25}
\em A hypergroup of two elements is an elementary abelian 2-hypergroup. If a hypergroup of two elements has a thick element of this hypergroup, it is neither a residually thin hypergroup nor a nilpotent hypergroup by Corollaries \ref{C;Corollary3.23} and \ref{C;Corollary3.24}.
\end{rem}
\begin{cor}\label{C;Corollary3.26}
Assume that $\mathbb{G}$ is a closed subset of $\mathbb{H}$. Then $\mathbb{G}$ is a residually thin closed subset of $\mathbb{H}$ if and only if $\mathbb{G}$ is a nilpotent closed subset of $\mathbb{H}$. In particular, $\mathbb{H}$ is a residually thin hypergroup if and only if $\mathbb{H}$ is a nilpotent hypergroup.
\end{cor}
\begin{proof}
The desired corollary follows from an application of Corollaries \ref{C;Corollary3.23}, \ref{C;Corollary3.24}.
\end{proof}
We conclude this section by giving an example of the main results of this section.
\begin{eg}\label{E;Example3.27}
\em Assume that $p=2$, $\mathbb{H}=\{e, q_1, q_2, r\}$, $q_1\in\mathrm{O}_\vartheta(\mathbb{H})$, and $q_2\notin\mathrm{O}_\vartheta(\mathbb{H})$. The hypermultiplication table of $\mathbb{H}$ with respect to the fixed operation $\circ$ is as follows:
\[
\begin{array}{c|cccc}
\mathbb{\circ} & e & q_1 & q_2 & r \\ \hline
e & \{e\} & \{q_1\} & \{q_2\} &\{r\} \\
q_1 & \{q_1\} & \{e\} &\{r\}  & \{q_2\}\\
q_2 & \{q_2\} & \{r\} & \{e, q_2\} & \{q_1, r\} \\
r & \{r\} & \{q_2\} & \{q_1, r\}& \{e, q_2\}\\
\end{array}
.\]
Then $p^\sharp=\mathrm{s}(\mathbb{H})=\mathrm{r}_2(\mathbb{H})=1$ by Lemma \ref{L;Lemma3.8}. Therefore $\mathrm{Thick}(\mathbb{H})=\mathrm{O}^\vartheta(\mathbb{H})=\{e, q_2\}$ and $\mathrm{Thin}(\mathbb{H})=\mathrm{O}_\vartheta(\mathbb{H})=\{e, q_1\}$ by combining Lemmas \ref{L;Lemma3.6},
\ref{L;Lemma3.7}, and the above data. According to Theorem \ref{T;TheoremA} and the above data, notice that all closed subsets
of $\mathbb{H}$ are precisely $\{e\}$, $\{e, q_1\}$, $\{e, q_2\}$, and $\mathbb{H}$. Moreover, Theorem \ref{T;TheoremB} and the above data imply that all strongly normal closed subsets of $\mathbb{H}$ are precisely $\{e, q_2\}$ and $\mathbb{H}$.
\end{eg}
\section{Isomorphisms and automorphism groups of closed subsets}
In this present section, we display a criterion for the isomorphic closed subsets of $\mathbb{H}$. As the other theme of this section, we present the automorphism groups
of all closed subsets of $\mathbb{H}$. As a preparation, we first introduce a sequence of six lemmas:
\begin{lem}\label{L;Lemma4.1}
Assume that $r, s\in\mathrm{Thick}(\mathbb{H})$. Then $\mathrm{s}(r)=\mathrm{s}(s)$ if and only if $\langle r\rangle\simeq\langle s\rangle$.
\end{lem}
\begin{proof}
As $\mathrm{s}(r)=0$ if and only if $r=e$, there is no loss to assume that $\mathrm{s}(r)\in\mathbb{N}$. For one direction, the hypotheses imply that the cardinalities of the supports of $r$ and $s$ are identical. Let $\alpha$ be a bijection from the support of $r$ to the support of $s$. By combining (E1), (E2), and Lemma \ref{L;Lemma3.1}, $\alpha$ induces an isomorphism from $\langle r\rangle$ to $\langle s\rangle$ that sends the element with the support $\mathbb{G}$ to the element with the support $\mathbb{G}^\alpha$ for any a subset $\mathbb{G}$ of the support of $r$. The desired lemma follows from Lemma \ref{L;Lemma3.1}.
\end{proof}
\begin{lem}\label{L;Lemma4.2}
Assume that $\mathbb{F}$ and $\mathbb{G}$ are thin closed subsets of $\mathbb{H}$. Then $\mathrm{r}_2(\mathbb{F})=\mathrm{r}_2(\mathbb{G})$ if and only if $\mathbb{F}\simeq\mathbb{G}$.
\end{lem}
\begin{proof}
For one direction, the equality $\mathrm{r}_2(\mathbb{F})\!=\!\mathrm{r}_2(\mathbb{G})$ and Lemma \ref{L;Lemma3.7} imply that there is a group isomorphism $\alpha$ from $\mathbb{F}^\gamma$ to $\mathbb{G}^\gamma$. For any $r\in \mathbb{F}$, let $r_\alpha$ be the element $s$ in $\mathbb{G}$ that satisfies the equality $\{s\}=\{r\}^\alpha$. So $\alpha$ induces an isomorphism from $\mathbb{F}$ to $\mathbb{G}$ that sends $r$ to $r_\alpha$ for any $r\in \mathbb{F}$. The
desired lemma follows from Lemma \ref{L;Lemma2.11}.
\end{proof}
\begin{lem}\label{L;Lemma4.3}
Assume that $\mathbb{F}$ and $\mathbb{G}$ are closed subsets of $\mathbb{H}$. Assume that $\mathrm{s}(\mathbb{F})\!=\!\mathrm{s}(\mathbb{G})$ and $\mathrm{r}_2(\mathbb{F})=\mathrm{r}_2(\mathbb{G})$. Then $\mathbb{F}\simeq\mathbb{G}$.
\end{lem}
\begin{proof}
By combining Theorem \ref{T;TheoremA}, Lemmas \ref{L;Lemma3.5}, \ref{L;Lemma4.1}, the equality $\mathrm{s}(\mathbb{F})\!=\!\mathrm{s}(\mathbb{G})$ shows that there is an isomorphism $\alpha$ from $\mathrm{O}^\vartheta(\mathbb{F})$ to $\mathrm{O}^\vartheta(\mathbb{G})$. By combining Lemmas \ref{L;Lemma3.4}, \ref{L;Lemma3.7},  \ref{L;Lemma4.2}, the equality $\mathrm{r}_2(\mathbb{F})=\mathrm{r}_2(\mathbb{G})$ shows that there is an isomorphism $\beta$ from the thin closed subset $\mathrm{O}_\vartheta(\mathbb{F})$ of $\mathbb{H}$ to the thin closed subset $\mathrm{O}_\vartheta(\mathbb{G})$ of $\mathbb{H}$. For any  $r\in \mathbb{F}$, Lemma \ref{L;Lemma3.6} lets $r_{(\alpha, \beta)}$ be the element $s$ in $\mathbb{G}$ satisfying $s^+=(r^+)^\alpha$ and $s^-\!=\!(r^-)^\beta$. By combining Lemmas \ref{L;Lemma3.3}, \ref{L;Lemma3.4}, \ref{L;Lemma3.6}, \ref{L;Lemma3.7}, \ref{L;Lemma3.8}, \ref{L;Lemma2.4}, $\alpha$ and $\beta$ induce an isomorphism from $\mathbb{F}$ to $\mathbb{G}$ that sends $r$ to $r_{(\alpha, \beta)}$ for any $r\in\mathbb{F}$. The desired lemma thus follows.
\end{proof}
\begin{lem}\label{L;Lemma4.4}
Assume that $\mathbb{F}$ and $\mathbb{G}$ are closed subsets of $\mathbb{H}$. Then $\mathbb{F}\simeq \mathbb{G}$ if and only if
$\mathrm{s}(\mathbb{F})=\mathrm{s}(\mathbb{G})$ and $\mathrm{r}_2(\mathbb{F})=\mathrm{r}_2(\mathbb{G})$.
\end{lem}
\begin{proof}
For one direction, let $\alpha\!\in\!\mathrm{Iso}(\mathbb{F}, \mathbb{G})$. As $\mathrm{O}_\vartheta(\mathbb{F})^\alpha\!\!\subseteq\!\! \mathrm{O}_\vartheta(\mathbb{G})$ and $\mathrm{O}_\vartheta(\mathbb{G})^{\alpha^{-1}}\!\!\subseteq\!\!\mathrm{O}_\vartheta(\mathbb{F})$,
notice that $|\mathrm{O}_\vartheta(\mathbb{F})|=|\mathrm{O}_\vartheta(\mathbb{G})|$ and $\mathrm{r}_2(\mathbb{F})=\mathrm{r}_2(\mathbb{G})$ by Lemma \ref{L;Lemma3.7}. As $|\mathbb{F}|=|\mathbb{G}|$, notice that
$\mathrm{s}(\mathbb{F})=\mathrm{s}(\mathbb{G})$ by Lemma \ref{L;Lemma3.8}. The desired lemma follows from Lemma \ref{L;Lemma4.3}.
\end{proof}
\begin{lem}\label{L;Lemma4.5}
Assume that $\mathbb{G}$ is a closed subset of $\mathbb{H}$. Assume that $s\in[ 0,\mathrm{s}(\mathbb{G})]$ and $r\in [ 0,\mathrm{r}_2(\mathbb{G})]$. Then there is a closed subset $\mathbb{F}$ of $\mathbb{H}$ such that $\mathbb{F}\subseteq\mathbb{G}$,
$\mathrm{s}(\mathbb{F})=s$, and $\mathrm{r}_2(\mathbb{F})=r$. Moreover, the number of all pairwise nonisomorphic closed subsets of $\mathbb{H}$ contained in $\mathbb{G}$ is equal to $\mathrm{s}(\mathbb{G})\mathrm{r}_2(\mathbb{G})+\mathrm{s}(\mathbb{G})+\mathrm{r}_2(\mathbb{G})+1$.
\end{lem}
\begin{proof}
By combining Theorem \ref{T;TheoremA}, Lemma \ref{L;Lemma3.1}, (E1), and (E2), there is $t\!\in\!\mathrm{Thick}(\mathbb{G})$ such that
$\mathrm{s}(t)=s$. According to Lemmas \ref{L;Lemma3.4} and \ref{L;Lemma3.7}, there is $\mathbb{E}\subseteq\mathbb{G}$ such that
$\mathbb{E}^\gamma$ is an elementary abelian 2-group of 2-rank $r$. Set $\mathbb{F}=\langle t\rangle\mathbb{E}$. The first statement is thus from combining Theorem \ref{T;TheoremA}, Lemmas \ref{L;Lemma3.6}, \ref{L;Lemma3.5}, \ref{L;Lemma3.2}, \ref{L;Lemma3.7}. The desired lemma follows from combining the first statement, Lemma \ref{L;Lemma4.4}, and a direct computation.
\end{proof}
\begin{lem}\label{L;Lemma4.6}
Assume that $\mathbb{G}$ is a closed subset of $\mathbb{H}$ and $\simeq_\mathbb{G}$ is the restriction of $\simeq$ to the set of
all closed subsets of $\mathbb{H}$ contained in $\mathbb{G}$. Then $\simeq_\mathbb{G}$ is an equivalence relation on the set of all
closed subsets of $\mathbb{H}$ contained in $\mathbb{G}$. Moreover, if $\mathbb{F}\subseteq\mathbb{G}$ and $\mathbb{F}$ is a closed subset of $\mathbb{H}$, then $\mathrm{s}(\mathbb{F})\in[0, \mathrm{s}(\mathbb{G})]$, $\mathrm{r}_2(\mathbb{F})\in [0,\mathrm{r}_2(\mathbb{G})]$, and the cardinality of the $\simeq_\mathbb{G}$-equivalence class containing $\mathbb{F}$ is equal to
$$\binom{\mathrm{s}(\mathbb{G})}{\mathrm{s}(\mathbb{F})}\binom{\mathrm{r}_2(\mathbb{G})}{\mathrm{r}_2(\mathbb{F})}_2.$$
\end{lem}
\begin{proof}
The first statement follows as $\simeq$ is an equivalence relation on the set of all closed subsets of $\mathbb{H}$. Since $\mathrm{O}^\vartheta(\mathbb{F})\subseteq\mathrm{O}^\vartheta(\mathbb{G})$ and
$\mathrm{O}_\vartheta(\mathbb{F})\subseteq\mathrm{O}_\vartheta(\mathbb{G})$ by Lemma \ref{L;Lemma2.9},
$\mathrm{s}(\mathbb{F})\in[0, \mathrm{s}(\mathbb{G})]$ and $\mathrm{r}_2(\mathbb{F})\in[0, \mathrm{r}_2(\mathbb{G})]$ by combining Lemmas
\ref{L;Lemma3.5}, \ref{L;Lemma3.6}, \ref{L;Lemma3.7}. By Lemma \ref{L;Lemma4.4}, it is enough to find all choices of a closed subset $\mathbb{E}$ of $\mathbb{H}$ that satisfies $\mathbb{E}\subseteq\mathbb{G}$, $\mathrm{s}(\mathbb{E})=\mathrm{s}(\mathbb{F})$, and $\mathrm{r}_2(\mathbb{E})=\mathrm{r}_2(\mathbb{F})$. The desired lemma follows from combining Theorem \ref{T;TheoremA}, (E1), (E2), Lemmas \ref{L;Lemma3.6}, \ref{L;Lemma3.5}, \ref{L;Lemma3.1}, \ref{L;Lemma3.2}, \ref{L;Lemma3.7}, and a direct computation.
\end{proof}
We are now ready to list the first main result of this section as the next theorem:
\begin{thm}\label{T;TheoremD}
Assume that $\mathbb{F}$ and $\mathbb{G}$ are closed subsets of $\mathbb{H}$. Then $\mathbb{F}\simeq\mathbb{G}$ if and only if
at least two equalities among $\mathrm{s}(\mathbb{F})=\mathrm{s}(\mathbb{G})$, $\mathrm{r}_2(\mathbb{F})=\mathrm{r}_2(\mathbb{G})$,
and $|\mathbb{F}|=|\mathbb{G}|$ hold. Moreover, there are exactly $pp^\sharp-(p^\sharp)^2+p+1$ pairwise distinct $\simeq$-equivalence classes. The cardinality of the $\simeq$-equivalence class containing $\mathbb{F}$ is equal to
$$\binom{p^\sharp}{\mathrm{s}(\mathbb{F})}\binom{p-p^\sharp}{\mathrm{r}_2(\mathbb{F})}_2.$$
\end{thm}
\begin{proof}
The desired theorem follows from combining Lemmas \ref{L;Lemma4.4}, \ref{L;Lemma4.5}, \ref{L;Lemma4.6}, \ref{L;Lemma3.8}.
\end{proof}
For some corollaries of Theorem \ref{T;TheoremD}, it is necessary to list the following lemmas:
\begin{lem}\label{L;Lemma4.8}
Assume that $\mathbb{G}$ is a closed subset of $\mathbb{H}$ and $r\in [0, \min\{\mathrm{s}(\mathbb{G}), \mathrm{r}_2(\mathbb{G})\}]$.
Then the number of all closed subsets of $\mathbb{H}$ of $2^r$ elements contained in $\mathbb{G}$ is equal to
$$\sum_{s=0}^{r}\binom{\mathrm{s}(\mathbb{G})}{s}\binom{\mathrm{r}_2(\mathbb{G})}{r-s}_2.$$
\end{lem}
\begin{proof}
By Lemmas \ref{L;Lemma3.8} and \ref{L;Lemma4.5}, let $\mathbb{F}$ be a closed subset of $\mathbb{H}$ of $2^r$ elements contained in $\mathbb{G}$. As $r\!\in\![0, \min\{\mathrm{s}(\mathbb{G}), \mathrm{r}_2(\mathbb{G})\}]$, Lemmas \ref{L;Lemma3.8} and \ref{L;Lemma4.6} imply that all possible choices of the pair $(\mathrm{s}(\mathbb{F}), \mathrm{r}_2(\mathbb{F}))$ are exactly $(0, r), (1, r-1), \ldots, (r, 0)$. The desired lemma follows from combining Lemma \ref{L;Lemma4.6}, Theorem \ref{T;TheoremD}, and a direct computation.
\end{proof}
\begin{lem}\label{L;Lemma4.9}
Assume that $\mathbb{G}$ is a closed subset of $\mathbb{H}$ and $r\in [\mathrm{s}(\mathbb{G})+1, \mathrm{r}_2(\mathbb{G})]$.
Then the number of all closed subsets of $\mathbb{H}$ of $2^r$ elements contained in $\mathbb{G}$ is equal to
$$\sum_{s=0}^{\mathrm{s}(\mathbb{G})}\binom{\mathrm{s}(\mathbb{G})}{s}\binom{\mathrm{r}_2(\mathbb{G})}{r-s}_2.$$
\end{lem}
\begin{proof}
By Lemmas \ref{L;Lemma3.8} and \ref{L;Lemma4.5}, let $\mathbb{F}$ be a closed subset of $\mathbb{H}$ of $2^r$ elements contained in $\mathbb{G}$. As $r\!\in\![\mathrm{s}(\mathbb{G})\!+\!1, \mathrm{r}_2(\mathbb{G})]$, Lemmas \ref{L;Lemma3.8} and \ref{L;Lemma4.6} imply that all possible choices of the pair $(\mathrm{s}(\mathbb{F}), \mathrm{r}_2(\mathbb{F}))$ are exactly $(0, r), (1, r-1), \ldots, (\mathrm{s}(\mathbb{G}), r\!-\!\mathrm{s}(\mathbb{G}))$.
The desired lemma follows from combining Lemma \ref{L;Lemma4.6}, Theorem \ref{T;TheoremD}, and a direct computation.
\end{proof}
\begin{lem}\label{L;Lemma4.10}
Assume that $\mathbb{G}$ is a closed subset of $\mathbb{H}$ and $r\in [\mathrm{r}_2(\mathbb{G})+1, \mathrm{s}(\mathbb{G})]$.
Then the number of all closed subsets of $\mathbb{H}$ of $2^r$ elements contained in $\mathbb{G}$ is equal to
$$\sum_{s=r-\mathrm{r}_2(\mathbb{G})}^{r}\binom{\mathrm{s}(\mathbb{G})}{s}\binom{\mathrm{r}_2(\mathbb{G})}{r-s}_2.$$
\end{lem}
\begin{proof}
By Lemmas \ref{L;Lemma3.8} and \ref{L;Lemma4.5}, let $\mathbb{F}$ be a closed subset of $\mathbb{H}$ of $2^r$ elements contained in $\mathbb{G}$. As $r\!\in\![\mathrm{r}_2(\mathbb{G})\!+\!1, \mathrm{s}(\mathbb{G})]$, Lemmas \ref{L;Lemma3.8} and \ref{L;Lemma4.6} imply that all possible choices of the pair $(\mathrm{s}(\mathbb{F}), \mathrm{r}_2(\mathbb{F}))$ are exactly $(r, 0), (r\!-\!1, 1),\ldots, (r-\mathrm{r}_2(\mathbb{G}), \mathrm{r}_2(\mathbb{G}))$. The desired lemma follows from combining Lemma \ref{L;Lemma4.6}, Theorem \ref{T;TheoremD}, and a direct computation.
\end{proof}
\begin{lem}\label{L;Lemma4.11}
Assume that $\mathbb{G}$ is a closed subset of $\mathbb{H}$ and $r\!\!\!\in\!\!\mathbb{N}_0\!\setminus\![0, \max\{\mathrm{s}(\mathbb{G}), \mathrm{r}_2(\mathbb{G})\}]$. Then the number of all closed subsets of $\mathbb{H}$ of $2^r$ elements contained in $\mathbb{G}$ is equal to $$\sum_{s=\min\{\mathrm{s}(\mathbb{G}), r-\mathrm{r}_2(\mathbb{G})\}}^{\max\{\mathrm{s}(\mathbb{G}), r-\mathrm{r}_2(\mathbb{G})\}}\binom{\mathrm{s}(\mathbb{G})}{s}\binom{\mathrm{r}_2(\mathbb{G})}{r-s}_2.$$
\end{lem}
\begin{proof}
By Lemma \ref{L;Lemma3.8}, there is no loss to let $r\in[\max\{\mathrm{s}(\mathbb{G}), \mathrm{r}_2(\mathbb{G})\}+1, \mathrm{s}(\mathbb{G})+\mathrm{r}_2(\mathbb{G})]$. By Lemma \ref{L;Lemma4.5}, let $\mathbb{F}$ be a closed subset of $\mathbb{H}$ of $2^r$ elements contained in $\mathbb{G}$. Lemmas \ref{L;Lemma3.8} and \ref{L;Lemma4.6} thus imply that all possible choices of the pair $(\mathrm{s}(\mathbb{F}), \mathrm{r}_2(\mathbb{F}))$ are exactly $(\mathrm{s}(\mathbb{G}), r-\mathrm{s}(\mathbb{G})), (\mathrm{s}(\mathbb{G})-1, r-\mathrm{s}(\mathbb{G})+1),\ldots, (r-\mathrm{r}_2(\mathbb{G}), \mathrm{r}_2(\mathbb{G}))$.
The desired lemma follows from combining Lemma \ref{L;Lemma4.6}, Theorem \ref{T;TheoremD}, and a direct computation.
\end{proof}
\begin{cor}\label{C;Corollary4.12}
Assume that $\mathbb{G}$ is a closed subset of $\mathbb{H}$. Assume that $r\in\mathbb{N}_0$.
Then the number of all closed subsets of $\mathbb{H}$ of $2^r$ elements contained in $\mathbb{G}$ is equal to
\[\begin{cases}
\sum_{s=0}^{r}\binom{\mathrm{s}(\mathbb{G})}{s}\binom{\mathrm{r}_2(\mathbb{G})}{r-s}_2, &\text{if}\ r\in [0, \min\{\mathrm{s}(\mathbb{G}), \mathrm{r}_2(\mathbb{G})\}],\\
\sum_{s=0}^{\mathrm{s}(\mathbb{G})}\binom{\mathrm{s}(\mathbb{G})}{s}\binom{\mathrm{r}_2(\mathbb{G})}{r-s}_2, &\text{if}\ r\in [\mathrm{s}(\mathbb{G})+1, \mathrm{r}_2(\mathbb{G})],\\
\sum_{s=r-\mathrm{r}_2(\mathbb{G})}^{r}\binom{\mathrm{s}(\mathbb{G})}{s}\binom{\mathrm{r}_2(\mathbb{G})}{r-s}_2, &\text{if}\ r\in [\mathrm{r}_2(\mathbb{G})+1, \mathrm{s}(\mathbb{G})],\\
\sum_{s=\min\{\mathrm{s}(\mathbb{G}), r-\mathrm{r}_2(\mathbb{G})\}}^{\max\{\mathrm{s}(\mathbb{G}), r-\mathrm{r}_2(\mathbb{G})\}}\binom{\mathrm{s}(\mathbb{G})}{s}\binom{\mathrm{r}_2(\mathbb{G})}{r-s}_2, &\text{if}\ r\in \mathbb{N}_0\setminus[0, \max\{\mathrm{s}(\mathbb{G}), \mathrm{r}_2(\mathbb{G})\}].
\end{cases}\]
As a particular case, the number of all closed subsets of $\mathbb{H}$ of $2^r$ elements is equal to
\[\begin{cases}
\sum_{s=0}^{r}\binom{p^\sharp}{s}\binom{p-p^\sharp}{r-s}_2, &\text{if}\ r\in [0, \min\{p^\sharp, p-p^\sharp\}],\\
\sum_{s=0}^{p^\sharp}\binom{p^\sharp}{s}\binom{p-p^\sharp}{r-s}_2, &\text{if}\ r\in [p^\sharp+1, p-p^\sharp],\\
\sum_{s=r-p+p^\sharp}^{r}\binom{p^\sharp}{s}\binom{p-p^\sharp}{r-s}_2, &\text{if}\ r\in [p-p^\sharp+1, p^\sharp],\\
\sum_{s=\min\{p^\sharp, r-p+p^\sharp\}}^{\max\{p^\sharp, r-p+p^\sharp\}}\binom{p^\sharp}{s}\binom{p-p^\sharp}{r-s}_2, &\text{if}\ r\in \mathbb{N}_0\setminus[0, \max\{p^\sharp, p-p^\sharp\}].
\end{cases}\]
\end{cor}
\begin{proof}
The first statement follows from combining Lemmas \ref{L;Lemma4.8}, \ref{L;Lemma4.9}, \ref{L;Lemma4.10}, and \ref{L;Lemma4.11}.
The desired corollary thus follows from Lemma \ref{L;Lemma3.8} and the above discussion.
\end{proof}
\begin{cor}\label{C;Corollary4.13}
Assume that $\mathbb{G}$ is a closed subset of $\mathbb{H}$ and $\mathbb{F}$ is a strongly normal closed subset of $\mathbb{G}$. Assume that $\mathbb{E}\subseteq\mathbb{G}$ and $\mathbb{E}$ is a closed subset of $\mathbb{H}$.
Then $\mathbb{E}\simeq\mathbb{F}$ if and only if $\mathbb{E}$ is a strongly normal closed subset of $\mathbb{G}$ of $|\mathbb{F}|$ elements. In particular, the $\simeq$-equivalence class containing a strongly normal
closed subset $\mathbb{D}$ of $\mathbb{H}$ contains precisely all strongly normal closed subsets of $\mathbb{H}$ of $|\mathbb{D}|$ elements.
\end{cor}
\begin{proof}
The desired corollary is from combining Lemmas \ref{L;Lemma3.10}, \ref{L;Lemma4.6}, Theorem \ref{T;TheoremD}.
\end{proof}
\begin{cor}\label{C;Corollary4.14}
Assume that $\mathbb{G}$ is a closed subset of $\mathbb{H}$. Assume that $r\in\mathbb{N}_0$. Then the number of all pairwise nonisomorphic strongly normal closed subsets of $\mathbb{G}$ is equal to $\mathrm{r}_2(\mathbb{G})+1$. Moreover, the number of all strongly normal closed subsets of $\mathbb{G}$ of $2^r$ elements is equal to
\[\begin{cases} 0, &\text{if}\ r\in [0, \mathrm{s}(\mathbb{G})-1],\\
\binom{\mathrm{r}_2(\mathbb{G})}{r-\mathrm{s}(\mathbb{G})}_2, &\text{if}\ r\in \mathbb{N}_0\setminus [0, \mathrm{s}(\mathbb{G})-1].
\end{cases}\]
As two particular cases of the above statements, there are exactly $p-p^\sharp+1$ pairwise distinct $\simeq$-equivalence classes containing the strongly normal closed subsets of $\mathbb{H}$. The number of all strongly normal closed subsets of $\mathbb{H}$ of $2^r$ elements is equal to
\[\begin{cases} 0, &\text{if}\ r\in [0, p^\sharp-1],\\
\binom{p-p^\sharp}{r-p^\sharp}_2, &\text{if}\ r\in \mathbb{N}_0\setminus [0, p^\sharp-1].
\end{cases}\]
\end{cor}
\begin{proof}
The first statement follows from combining Lemmas \ref{L;Lemma3.10}, \ref{L;Lemma4.6}, \ref{L;Lemma4.5}, Theorem \ref{T;TheoremD}.
By Lemmas \ref{L;Lemma3.8} and \ref{L;Lemma3.10}, there is no loss to let $r\in [\mathrm{s}(\mathbb{G}), \mathrm{s}(\mathbb{G})+\mathrm{r}_2(\mathbb{G})]$. The second statement follows from combining Corollary \ref{C;Corollary4.13}, Theorem \ref{T;TheoremD}, Lemmas \ref{L;Lemma3.10}, \ref{L;Lemma4.6}, \ref{L;Lemma3.8}. The desired corollary follows from Lemma \ref{L;Lemma3.8} and the above discussion.
\end{proof}
For the other theme of this section, we list the following lemmas as a preparation:
\begin{lem}\label{L;Lemma4.15}
Assume that $\mathbb{F}$ and $\mathbb{G}$ are closed subsets of $\mathbb{H}$. Assume that $\alpha$ is an injective homomorphism from $\mathbb{F}$ to $\mathbb{G}$. Then $\mathrm{O}^\vartheta(\mathbb{F})^\alpha\subseteq\mathrm{O}^\vartheta(\mathbb{G})$ and $\mathrm{O}_\vartheta(\mathbb{F})^\alpha\subseteq\mathrm{O}_\vartheta(\mathbb{G})$.
\end{lem}
\begin{proof}
Theorem \ref{T;TheoremA} implies that $\mathrm{O}^\vartheta(\mathbb{F})=\langle r\rangle$ for some $r\in\mathrm{Thick}(\mathbb{F})$. There is no loss to require that $\mathrm{s}(r)\in\mathbb{N}$ and $\{q_1, q_2, \ldots, q_{\mathrm{s}(r)}\}$ is the support of $r$. Lemma \ref{L;Lemma3.1} thus implies that $\langle r\rangle=\langle q_1, q_2, \ldots, q_{\mathrm{s}(r)}\rangle$. Pick $s\in\{q_1, q_2, \ldots, q_{\mathrm{s}(r)}\}$. Notice that $s^\alpha$ is an involution of $\mathbb{H}$ as $s^\alpha\neq e$ and $\{e, s^\alpha\}$ is a closed subset of $\mathbb{H}$. Hence the combination of Theorem \ref{T;TheoremA}, Lemmas \ref{L;Lemma3.1}, \ref{L;Lemma2.10}, \ref{L;Lemma2.5}, (H2) implies that $s^\alpha\in\mathrm{O}^\vartheta(\mathbb{G})$. The desired lemma thus follows as $s$ is chosen from $\{q_1, q_2, \ldots, q_{\mathrm{s}(r)}\}$ arbitrarily.
\end{proof}
For presenting the remaining lemmas, it is necessary to list the following notation:
\begin{nota}\label{N;Notation4.16}
\em Assume that $\mathbb{F}$ and  $\mathbb{G}$ are closed subsets of $\mathbb{H}$. Assume that $\alpha$ is an injective homomorphism
from $\mathbb{F}$ to $\mathbb{G}$. Let $\alpha^+$ be the restriction of $\alpha$ to $\mathrm{O}^\vartheta(\mathbb{F})$. Let $\alpha^-$ be the restriction of $\alpha$ to $\mathrm{O}_\vartheta(\mathbb{F})$. As $\mathrm{Thick}(\mathbb{F})=\mathrm{O}^\vartheta(\mathbb{F})$ and $\mathrm{Thin}(\mathbb{F})=\mathrm{O}_\vartheta(\mathbb{F})$ by Lemmas \ref{L;Lemma3.6} and \ref{L;Lemma3.7}, Lemma \ref{L;Lemma3.8} implies that $(r^+)^{\alpha^+}(r^-)^{\alpha^-}\!=\!\{r^\alpha\}$ for any $r\in\mathbb{F}$. As Lemmas \ref{L;Lemma3.4} and \ref{L;Lemma3.7} imply that $\mathrm{O}_\vartheta(\mathbb{F})$ is a thin closed subset of $\mathbb{H}$, Lemma \ref{L;Lemma4.15} thus implies that  $\alpha^+\in\mathrm{Aut}(\mathrm{O}^\vartheta(\mathbb{F}))$ and $\alpha^-\in\mathrm{Aut}(\mathrm{O}_\vartheta(\mathbb{F}))$ if $\mathbb{F}=\mathbb{G}$ and $\alpha\in\mathrm{Aut}(\mathbb{F})$.
\end{nota}
\begin{lem}\label{L;Lemma4.17}
Assume that $\mathbb{E}$, $\mathbb{F}$, $\mathbb{G}$ are closed subsets of $\mathbb{H}$. Assume that $\alpha$ is an injective homomorphism from $\mathbb{E}$ to $\mathbb{F}$. Assume that $\beta$ is an injective homomorphism from $\mathbb{F}$ to $\mathbb{G}$.
Then $\alpha\beta$ is an injective homomorphism from $\mathbb{E}$ to $\mathbb{G}$, $(\alpha\beta)^+=\alpha^+\beta^+$, and $(\alpha\beta)^-=\alpha^-\beta^-$.
\end{lem}
\begin{proof}
The first statement is from the above hypotheses and a direct computation. Notice that  $\mathrm{O}^\vartheta(\mathbb{E})^\alpha\subseteq\mathrm{O}^\vartheta(\mathbb{F})$ and $\mathrm{O}_\vartheta(\mathbb{E})^\alpha\subseteq\mathrm{O}_\vartheta(\mathbb{F})$ by Lemma \ref{L;Lemma4.15}.
The desired lemma thus follows from the above discussion and a direct computation.
\end{proof}
\begin{lem}\label{L;Lemma4.18}
Assume that $\mathbb{F}$ and $\mathbb{G}$ are closed subsets of $\mathbb{H}$. Assume that $\alpha$ and $\beta$ are injective homomorphisms from $\mathbb{F}$ to $\mathbb{G}$. Then $\alpha^+$ is an injective homomorphism from $\mathrm{O}^\vartheta(\mathbb{F})$ to $\mathrm{O}^\vartheta(\mathbb{G})$ and $\alpha^-$ is an injective homomorphism from the thin closed subset $\mathrm{O}_\vartheta(\mathbb{F})$ of $\mathbb{H}$ to the thin closed subset $\mathrm{O}_\vartheta(\mathbb{G})$ of $\mathbb{H}$. Moreover, $\alpha=\beta$ if and only if $\alpha^+=\beta^+$ and $\alpha^-=\beta^-$.
\end{lem}
\begin{proof}
As Lemmas \ref{L;Lemma3.4} and \ref{L;Lemma3.7} imply that $\mathrm{O}_\vartheta(\mathbb{F})$ and $\mathrm{O}_\vartheta(\mathbb{G})$ are thin closed subsets of $\mathbb{H}$, the first statement follows from Lemma \ref{L;Lemma4.15} and the above hypotheses. The desired lemma follows as $\!(r^+)^{\alpha^+}(r^-)^{\alpha^-}\!=\!\{r^\alpha\}$ and $\!(r^+)^{\beta^+}(r^-)^{\beta^-}\!=\!\{r^\beta\}$ if $r\in\mathbb{F}$.
\end{proof}
\begin{lem}\label{L;Lemma4.19}
Assume that $\mathbb{F}$ and $\mathbb{G}$ are closed subsets of $\mathbb{H}$. Assume that $\alpha$ is an injective homomorphism from $\mathrm{O}^\vartheta(\mathbb{F})$ to $\mathrm{O}^\vartheta(\mathbb{G})$. Assume that $\beta$ is an injective homomorphism from the thin
closed subset $\mathrm{O}_\vartheta(\mathbb{F})$ of $\mathbb{H}$ to the thin closed subset $\mathrm{O}_\vartheta(\mathbb{G})$ of $\mathbb{H}$. Then there is a unique injective homomorphism from $\mathbb{F}$ to $\mathbb{G}$ such that its restriction to $\mathrm{O}^\vartheta(\mathbb{G})$ is $\alpha$ and its restriction to $\mathrm{O}_\vartheta(\mathbb{G})$ is $\beta$.
\end{lem}
\begin{proof}
If $r\in \mathbb{F}$, Lemma \ref{L;Lemma3.6} lets $r_{(\alpha, \beta)}$ be the element $s$ in $\mathbb{G}$ satisfying $s^+=(r^+)^\alpha$ and $s^-\!=\!(r^-)^\beta$. As Lemmas \ref{L;Lemma3.4} and \ref{L;Lemma3.7} imply that both $\mathrm{O}_\vartheta(\mathbb{F})$ and $\mathrm{O}_\vartheta(\mathbb{G})$ are thin closed subsets of $\mathbb{H}$, Lemma \ref{L;Lemma3.8} gives a map $\alpha\otimes\beta$ from $\mathbb{F}$ to $\mathbb{G}$ that sends $r$ to $r_{(\alpha,\beta)}$ for any $r\in\mathbb{F}$. The combination of Lemmas \ref{L;Lemma3.3}, \ref{L;Lemma3.4}, \ref{L;Lemma3.6}, \ref{L;Lemma3.7}, a direct computation shows that
$\alpha\otimes\beta$ is a homomorphism from $\mathbb{F}$ to $\mathbb{G}$. By Lemmas \ref{L;Lemma3.8} and \ref{L;Lemma2.4},
notice that $\alpha\otimes\beta$ is also an injective homomorphism from $\mathbb{F}$ to $\mathbb{G}$. Moreover, notice that
$(\alpha\otimes\beta)^+=\alpha$ and $(\alpha\otimes\beta)^-=\beta$. The desired lemma follows from Lemma \ref{L;Lemma4.18}.
\end{proof}
\begin{lem}\label{L;Lemma4.20}
Assume that $\mathbb{G}$ is a closed subset of $\mathbb{H}$. Assume that $r\in\mathrm{Thick}(\mathbb{G})$. Then
$\mathrm{Aut}(\langle r\rangle)\cong\mathbb{S}_{\mathrm{s}(r)}$ and $\mathrm{Aut}(\mathrm{O}^\vartheta(\mathbb{G}))\cong\mathbb{S}_{\mathrm{s}(\mathbb{G})}$. In particular, $\mathrm{Aut}(\mathrm{O}^\vartheta(\mathbb{H}))\cong\mathbb{S}_{p^\sharp}$.
\end{lem}
\begin{proof}
There is no loss to require that $\mathrm{s}(r)\in\mathbb{N}$ and $\{q_1, q_2, \ldots, q_{\mathrm{s}(r)}\}$ is the support of $r$. Lemma \ref{L;Lemma3.1} thus implies that $\langle r\rangle=\langle q_1, q_2, \ldots, q_{\mathrm{s}(r)}\rangle$. Pick $\alpha\in\mathrm{Aut}(\langle r\rangle)$ and $s\in\{q_1, q_2, \ldots, q_{\mathrm{s}(r)}\}$. Notice that $s^\alpha$ is an involution of $\mathbb{H}$ as $s^\alpha\neq e$ and $\{e, s^\alpha\}$ is a closed subset of $\mathbb{H}$. Moreover, notice that $s^\alpha\in\{q_1, q_2, \ldots, q_{\mathrm{s}(r)}\}$ by Lemma \ref{L;Lemma3.1}. As $\alpha$ and $s$ are arbitrarily chosen from $\mathrm{Aut}(\langle r\rangle)$
and $\{q_1, q_2, \ldots, q_{\mathrm{s}(r)}\}$, the combination of Lemma \ref{L;Lemma3.1}, (E1), (E2) implies that $\mathrm{Aut}(\langle r\rangle)$ acts on $\{q_1, q_2, \ldots, q_{\mathrm{s}(r)}\}$ faithfully.

Let $\beta$ be a bijection from $\{q_1, q_2, \ldots, q_{\mathrm{s}(r)}\}$ to $\{q_1, q_2, \ldots, q_{\mathrm{s}(r)}\}$. By combining (E1), (E2), Lemma \ref{L;Lemma3.1}, $\beta$ induces an isomorphism from $\langle r\rangle$ to $\langle r\rangle$ that sends the element with the support $\mathbb{F}$ to the element with the support $\mathbb{F}^\beta$ for any $\mathbb{F}\subseteq\{ q_1, q_2, \ldots, q_{\mathrm{s}(r)}\}$. As $\beta$ is an arbitrarily chosen bijection from $\{q_1, q_2, \ldots, q_{\mathrm{s}(r)}\}$ to $\{q_1, q_2, \ldots, q_{\mathrm{s}(r)}\}$, the first formula follows. The second formula thus follows from combining Lemmas
\ref{L;Lemma3.6}, \ref{L;Lemma3.5}, and the first formula. The desired lemma thus follows from Lemma \ref{L;Lemma3.8}.
\end{proof}
We are now ready to give the final main result of this section as the next theorem:
\begin{thm}\label{T;TheoremE}
Assume that $\mathbb{G}$ is a closed subset of $\mathbb{H}$. Then
$$\mathrm{Aut}(\mathbb{G})\cong\mathbb{S}_{\mathrm{s}(\mathbb{G})}\times \mathbb{GL}(\mathrm{r}_2(\mathbb{G}), 2)\ \text{and}\ \mathrm{Aut}(\mathbb{H})\cong\mathbb{S}_{p^\sharp}\times \mathbb{GL}(p-p^\sharp, 2).$$
\end{thm}
\begin{proof}
By Lemmas \ref{L;Lemma3.4} and \ref{L;Lemma3.7}, $\mathrm{O}_\vartheta(\mathbb{G})$ is a thin closed subset of $\mathbb{H}$. By combining Lemmas \ref{L;Lemma4.17}, \ref{L;Lemma4.18}, \ref{L;Lemma4.19}, the map from $\mathrm{Aut}(\mathbb{G})$ to $\mathrm{Aut}(\mathrm{O}^\vartheta(\mathbb{G}))\times\mathrm{Aut}(\mathrm{O}_\vartheta(\mathbb{G}))$ that sends
$\alpha$ to $(\alpha^+, \alpha^-)$ for any $\alpha\in\mathrm{Aut}(\mathbb{G})$ is an obvious group isomorphism. The desired theorem thus follows from combining Lemmas \ref{L;Lemma4.20}, \ref{L;Lemma2.12}, \ref{L;Lemma2.10}, \ref{L;Lemma3.7}, \ref{L;Lemma3.8}.
\end{proof}
For some corollaries of Theorem \ref{T;TheoremE}, it is necessary to offer the following lemmas:
\begin{lem}\label{L;Lemma4.22}
Assume that $\mathbb{G}$ is a closed subset of $\mathbb{H}$. Then $\mathrm{Aut}(\mathbb{G})$ is isomorphic to the trivial group if and only if $(\mathrm{s}(\mathbb{G}), \mathrm{r}_2(\mathbb{G}))\in\{(0,0), (0,1), (1, 0), (1,1)\}$.
\end{lem}
\begin{proof}
The desired lemma follows from Theorem \ref{T;TheoremE} and a direct computation.
\end{proof}
\begin{lem}\label{L;Lemma4.23}
Assume that $\mathbb{G}$ is a closed subset of $\mathbb{H}$. Then $\mathrm{Aut}(\mathbb{G})\cong \mathbb{S}_3$ if and only if
$(\mathrm{s}(\mathbb{G}), \mathrm{r}_2(\mathbb{G}))\in\{(0,2), (1,2), (3, 0), (3,1)\}$.
\end{lem}
\begin{proof}
The desired lemma follows from Theorem \ref{T;TheoremE} and a direct computation.
\end{proof}
\begin{lem}\label{L;Lemma4.24}
Assume that $\mathbb{G}$ is a closed subset of $\mathbb{H}$. Then there are subgroups $\mathbb{E}$ and $\mathbb{F}$ of $\mathrm{Aut}(\mathbb{G})$ such that $\mathrm{Aut}(\mathbb{G})\cong\mathbb{E}\times\mathbb{F}$ and neither $\mathbb{E}$ nor $\mathbb{F}$ is a direct product of its nontrivial proper subgroups.
\end{lem}
\begin{proof}
By Theorem \ref{T;TheoremE}, it is enough to check that neither $\mathbb{S}_{\mathrm{s}(\mathbb{G})}$ nor $\mathbb{GL}(\mathrm{r}_2(\mathbb{G}), 2)$ is a direct product of its nontrivial proper subgroups. There is no loss to assume that  $\mathrm{s}(\mathbb{G})\in\mathbb{N}\setminus[1, 4]$. As the center of $\mathbb{S}_{\mathrm{s}(\mathbb{G})}$ is its trivial subgroup,  $\mathbb{S}_{\mathrm{s}(\mathbb{G})}$ is not a direct product of its alternating subgroup and a subgroup of two elements. Assume that $\mathbb{S}_{\mathrm{s}(\mathbb{G})}$ is a direct product of its nontrivial proper subgroups. The simplicity of the alternating subgroup of $\mathbb{S}_{\mathrm{s}(\mathbb{G})}$ shows that $\mathbb{S}_{\mathrm{s}(\mathbb{G})}$ is an elementary abelian  2-group. This is absurd. Notice that $\mathbb{GL}(\mathrm{r}_2(\mathbb{G}), 2)$ is a simple group if and only if $\mathrm{r}_2(\mathbb{G})\!\notin\![0,2]$. Notice that $\mathbb{GL}(2, 2)\!\cong\!\mathbb{S}_3$. The desired lemma follows from the above discussion.
\end{proof}
\begin{lem}\label{L;Lemma4.25}
Assume that $\mathbb{G}$ is a closed subset of $\mathbb{H}$. Then $\mathrm{Aut}(\mathbb{G})$ is isomorphic to a direct product of the symmetric groups if and only if $\mathrm{r}_2(\mathbb{G})\in[0,2]$.
\end{lem}
\begin{proof}
As $\mathbb{GL}(0, 2)\!\cong\!\mathbb{GL}(1, 2)\!\cong\!\mathbb{S}_1$ and $\mathbb{GL}(2, 2)\!\cong\!\mathbb{S}_3$, the desired lemma thus follows from combining Theorem \ref{T;TheoremE}, Lemma \ref{L;Lemma4.24}, and the Krull-Schmidt Theorem.
\end{proof}
\begin{cor}\label{C;Corollary4.26}
Assume that $\mathbb{F}$ and $\mathbb{G}$ are closed subsets of $\mathbb{H}$. Assume that neither the trivial group nor $\mathbb{S}_3$ is isomorphic to $\mathrm{Aut}(\mathbb{F})$. Then $\mathrm{Aut}(\mathbb{F})\!\cong\!\mathrm{Aut}(\mathbb{G})$ if and only if $\mathbb{F}\simeq\mathbb{G}$.
\end{cor}
\begin{proof}
It is enough to check that $\mathrm{Aut}(\mathbb{F})\cong\mathrm{Aut}(\mathbb{G})$ only if $\mathbb{F}\simeq\mathbb{G}$. The combination of Theorem \ref{T;TheoremE}, Lemmas \ref{L;Lemma4.24}, \ref{L;Lemma4.25}, and the Krull-Schmidt Theorem implies that $\mathrm{Aut}(\mathbb{F})\cong\mathbb{S}_3\times\mathbb{S}_3$ if and only if $(\mathrm{s}(\mathbb{F}), \mathrm{r}_2(\mathbb{F}))\!=\!(3, 2)$. By Theorem \ref{T;TheoremD}, there is no loss to assume further that $(\mathrm{s}(\mathbb{F}), \mathrm{r}_2(\mathbb{F}))\!\neq\!(3, 2)$. As neither the trivial group nor $\mathbb{S}_3$ is isomorphic to $\mathrm{Aut}(\mathbb{F})$, the combination of Lemmas \ref{L;Lemma4.22}, \ref{L;Lemma4.23}, and the above discussion thus implies that $\mathrm{s}(\mathbb{F})\in\mathbb{N}\setminus\{1, 3\}$ or $\mathrm{r}_2(\mathbb{F})\in\mathbb{N}\setminus[1, 2]$. So the combination of Theorem \ref{T;TheoremE}, Lemmas \ref{L;Lemma4.24}, \ref{L;Lemma4.25}, and the Krull-Schmidt Theorem implies that $\mathrm{s}(\mathbb{F})=\mathrm{s}(\mathbb{G})$ and $\mathrm{r}_2(\mathbb{F})=\mathrm{r}_2(\mathbb{G})$. The desired corollary follows from Theorem \ref{T;TheoremD}.
\end{proof}
\begin{cor}\label{C;Corollary4.27}
Assume that $\mathbb{G}$ is a closed subset of $\mathbb{H}$. Assume that $\mathbb{F}$ is a strongly normal closed subset of $\mathbb{G}$. Then $\mathrm{Aut}(\mathbb{F})\!\cong\!\mathbb{S}_{\mathrm{s}(\mathbb{G})}\times \mathbb{GL}(\mathrm{r}_2(\mathbb{F}), 2)$. In particular, assume that $\mathbb{E}$ is a strongly normal closed subset of $\mathbb{H}$. Then $\mathrm{Aut}(\mathbb{E})\cong\mathbb{S}_{p^\sharp}\times \mathbb{GL}(\mathrm{r}_2(\mathbb{E}), 2)$.
\end{cor}
\begin{proof}
The first statement follows from combining Lemmas \ref{L;Lemma3.10}, \ref{L;Lemma4.6}, and Theorem \ref{T;TheoremE}.
The desired corollary thus follows from the first statement and Lemma \ref{L;Lemma3.8}.
\end{proof}
\begin{cor}\label{C;Corollary4.28}
Assume that $\mathbb{F}$ and $\mathbb{G}$ are closed subsets of $\mathbb{H}$. Assume that $\mathbb{F}\subseteq\mathbb{G}$.
Assume that $\mathbb{E}$ is a strongly normal closed subset of $\mathbb{G}$ and neither the trivial group nor $\mathbb{S}_3$ is isomorphic to $\mathrm{Aut}(\mathbb{E})$. Then $\mathrm{Aut}(\mathbb{E})\!\cong\!\mathrm{Aut}(\mathbb{F})$ if and only if $\mathbb{F}$ is a strongly normal closed subset of $\mathbb{G}$ and $\mathbb{E}\!\simeq\! \mathbb{F}$. In particular,
if $\mathbb{D}$ is a strongly normal closed subset of $\mathbb{H}$ and neither the trivial group nor $\mathbb{S}_3$ is isomorphic to $\mathrm{Aut}(\mathbb{D})$, then $\mathrm{Aut}(\mathbb{D})\!\!\cong\!\!\mathrm{Aut}(\mathbb{G})$ if and only if
$\mathbb{G}$ is a strongly normal closed subset of $\mathbb{H}$ and $\mathbb{D}\!\simeq\! \mathbb{G}$.
\end{cor}
\begin{proof}
The desired corollary follows from an application of Corollaries \ref{C;Corollary4.26}, \ref{C;Corollary4.13}.
\end{proof}
We close this paper by presenting an example of the main results of this section.
\begin{eg}\label{E;Example4.29}
\em Assume that $p=2$, $\mathbb{H}=\{e, q_1, q_2, r\}$, $q_1\in\mathrm{O}_\vartheta(\mathbb{H})$, and $q_2\notin\mathrm{O}_\vartheta(\mathbb{H})$. Example \ref{E;Example3.27} shows that $\{e\}$, $\{e, q_1\}$, $\{e, q_2\}$, $\mathbb{H}$ are exactly all closed subsets of $\mathbb{H}$. By Theorem \ref{T;TheoremD}, notice that $\{e\}$, $\{e, q_1\}$, $\{e, q_2\}$, $\mathbb{H}$ are pairwise nonisomorphic closed subsets of $\mathbb{H}$. Moreover, Theorem \ref{T;TheoremE} implies that all automorphism groups of the closed subsets $\{e\}$, $\{e, q_1\}$, $\{e, q_2\}$, $\mathbb{H}$ of $\mathbb{H}$ are isomorphic to the trivial group.
\end{eg}
\subsection*{Disclosure statement} No relevant financial or nonfinancial interests are reported.
\subsection*{Data availability statement} All used data are contained in this written paper.

\end{document}